\documentclass[journal]{IEEEtran}  

\usepackage{cancel}
\usepackage{cite}
\usepackage{graphics}
\usepackage{multirow}
\usepackage[table]{xcolor}
\usepackage{epsfig}
\usepackage{mathptmx} 
\usepackage{times} 
\usepackage{amsmath}
\usepackage{amssymb,,amsfonts} 
\usepackage{amsthm} 
\usepackage{mathptmx}
\usepackage{dblfloatfix} 

\newtheorem{definition}{Definition}
\newtheorem{theorem}{\bf{Theorem}}
\newtheorem{remark}{\bf{Remark}}
\usepackage{comment}
\usepackage{mathtools}
\usepackage{tabularx} 
\usepackage{multicol}
\usepackage{xcolor}

\usepackage[section]{placeins}
\usepackage{cuted}

\usepackage{tabularray}
\usepackage{float}
\usepackage[shortlabels]{enumitem}
\usepackage{placeins}
\DeclareMathAlphabet{\mathcal}{OMS}{cmsy}{m}{n}

\theoremstyle{remark}

\theoremstyle{definition}
\usepackage[utf8]{inputenc}
\usepackage[english]{babel}

\usepackage{multirow}
\usepackage{tabularx}
\usepackage{adjustbox} 

\newcommand{\PreserveBackslash}[1]{\let\temp=\\#1\let\\=\temp}
\newcolumntype{C}[1]{>{\PreserveBackslash\centering}p{#1}}
\newcolumntype{R}[1]{>{\PreserveBackslash\raggedleft}p{#1}}
\newcolumntype{L}[1]{>{\PreserveBackslash\raggedright}p{#1}}

\newcommand{\sina}[1]{\textcolor{blue}{#1}}
\newcommand{\nrc}[1]{\textcolor{blue}{#1}}

\newcommand{\lk}[1]{\textcolor{blue}{#1}}

\def\cl#1{\textcolor{teal}{#1}}

\begin{document}

\title{Robust Adaptive Supplementary Control for Damping Weak-Grid SSOs Involving IBRs}

\author{Sina~Ameli,
        Lilan~Karunaratne,~\IEEEmembership{Student Member,~IEEE},  Nilanjan Ray Chaudhuri,~\IEEEmembership{Senior Member,~IEEE}, and  Constantino~Lagoa,~\IEEEmembership{Member,~IEEE \vspace{-10pt}}

% \thanks{This work is supported under Agreement 37532 by the Advanced Grid Modeling (AGM) Program of the Office of Electricity, the Department of Energy (DOE).}
\thanks{S. Ameli, L. Karunaratne, N. R. Chaudhuri and C. M. Lagoa are with the School of Electrical Engineering and Computer
Science, Penn State University, University Park, PA, USA e-mail: sba6099@psu.edu, lvk5363@psu.edu, nuc88@psu.edu, and cml18@psu.edu.}
\thanks{Financial support from NSF Grant Award ECCS 2317272 is gratefully acknowledged.}
}

\maketitle

\begin{abstract}
Subsynchronous oscillations (SSOs) involving grid-following converters (GFLCs) connected to weak grids are a relatively new phenomena observed in modern power systems. SSOs are further exacerbated when grids become weaker because lines are disconnected due to maintenance or following faults. Such undesirable oscillations have also led to curtailment of inverter-based resource (IBR) outputs. In contrast to most literature addressing the issue by retuning/redesigning of standard IBR controllers, we propose a robust adaptive supplementary control for damping of such SSOs while keeping standard controls unaltered. As a result, uncertainty in system conditions can be handled without negatively impacting the nominal IBR performance. To that end, the adaptive control law is derived for a GFLC connected to the grid, where the grid is modeled by the Thevenin's equivalent representation with uncertainty and disturbances. The theoretical result provides dissipativity certificate for the closed-loop error dynamics with sufficient conditions for stability. The effectiveness of the developed controller is validated with several case studies conducted on a single-GFLC-infinite-bus test system, the IEEE $2$-area test system, wherein some of the synchronous generators are replaced by GFLCs, \nrc{and a modified IEEE $5$-area test system with two GFLCs}. The findings demonstrate that under very weak grid conditions, the proposed robust adaptive control performs well in stabilizing SSO modes, which a classical state-feedback control method fails to address. 
\end{abstract}


\IEEEpeerreviewmaketitle

\vspace{-10pt}
\section{Introduction}\label{sec:Intro}
\IEEEPARstart{R}{apid} growth of inverter-based resources (IBRs) has pushed the contemporary power system into an uncharted territory. IBRs are based on two types of converters namely, grid-following converters (GFLCs)~\cite{yazdani2010voltage} and grid-forming converters (GFCs)\cite{GFC_IEE_Access}. GFLCs are widely adopted in power systems to integrate renewable sources to the grid. 
In recent times, several incidents of  subsynchronous oscillations (SSOs) involving GFLCs have been reported in different parts of the world \cite{IBRSSOTF}, which can be classified into three types: (a) series capacitor SSO, (b) weak-grid SSO, and (c) inter-IBR SSO \cite{Lingling-InterIBR}. The focus of this paper is stabilization/damping of weak-grid SSOs.

\subsection{Literature Review on Weak-Grid SSO Stabilization}
Weak-grid SSO phenomena have been analyzed in multiple papers including    \cite{IBRSSOTF, Lingling-19-Type4WindModel,Lingling-21-ReducedAnalyticalPV,Lingling-23-NewIBRoscType,SSO_ibr}. \nrc{It has been recognized that typically the following factors influence weak-grid SSOs \cite{IBRSSOTF, Lingling-19-Type4WindModel,Lingling-21-ReducedAnalyticalPV,Lingling-23-NewIBRoscType,SSO_ibr}:} (a) grid strength quantified by short circuit ratio (SCR) or its variants and operating conditions influenced by factors including IBR power outputs, and (b) \nrc{IBR controllers including Phase-Lock-Loop (PLL),} inner current controller, outer voltage controller, and communication delay involved in plant-level voltage controls. Based upon this understanding, certain approaches have been proposed/implemented to stabilize SSOs. We can divide the weak-grid SSO stabilization approaches into two broad categories: (i) grid strength-focused, and (ii) IBR-focused.

\textit{(i) Grid strength-focused measures \cite{Ignacio-18-ChinaSSO,Jian-20-Type3Model,Chun-18-PVosc}:} In northeast China, a $13$-Hz SSO was avoided by adding a $500$-kV line to improve grid strength (in addition to wind farm control retuning). A $20$-Hz SSO in Canada was mitigated by closing a tie-breaker to increase system strength.

\textit{(ii) IBR-focused measures:} The IBR-focused measures can be subdivided into: (ii.A) temporary IBR power curtailment to alleviate SSO; (ii.B) permanent controller retuning by changing controller parameters while keeping control structure unaltered; (ii.C) controller retuning and adding supplementary modulation control; (ii.D) adding supplementary modulation to existing controls; and (ii.E) controller resynthesis by altering controller structure. 

\underline{(ii.A) and (ii.B)}: In most instances of SSOs in the field, \textit{temporary power curtailment} was performed by the system operator \textit{followed by a \textcolor{black}{permanent} controller retuning} by the vendors \cite{Huang-12-ERCOT,E3C19,NERCguide17,Liu-17-SSCI}. \nrc{Such retuning efforts largely focused on PLL retuning including bandwidth reduction} and outer voltage controller parameter change \cite{Huang-12-ERCOT,E3C19,NERCguide17,Liu-17-SSCI,IBRSSOTF}. In addition, research has shown that the proportional gain of the PI controller used in the inner current control loops can be retuned to avoid adverse impact of the weak-grid SSO mode \cite{JunLiang_VSI_WeakGrid,Lennart_PLL_WeakGrid}. 

\underline{(ii.C)}: Reference \cite{mittal2024stability} proposes a two-step process for SSO stabilization in a system with two parallel IBRs. In the first step, the best combination of IBR control modes in the outer loops have been determined followed by appropriate controller tuning. In the second step, a supplementary control involving voltage at the \sina{point of interconnection} \sina{ (}POI\sina{)} is used as a feedback signal through a high-pass filter to modulate the reactive power reference of one of the IBRs. 

\begin{table*}[ht!]
\lk{
\caption{Pros and Cons of Control Solutions}
\label{Tab:Methods}
% \vspace{8pt}
\centering
\begin{tabular}{|L{6cm}| *{7}{C{1.5cm}|}}
    \hline
    \centerline{\textbf{Attributes}} & \textbf{\hspace{-1pt}Robustness guarantee} & \textbf{Analytical stability guarantee} & \textbf{Preserves nominal performance} & \textbf{Minimally invasive} & \textbf{Adaptability} & \textbf{Suitability for practical application}\\ %& \textbf{Platform used}\\
    \hline
    Standard controller retuning~ \cite{Huang-12-ERCOT,E3C19,NERCguide17,Liu-17-SSCI,IBRSSOTF,JunLiang_VSI_WeakGrid,Lennart_PLL_WeakGrid} & $\times$: no theoretical guarantees  & $\times$: no rigorous stability guarantees & $\times$: introduces conservatism & $\checkmark$: only changes controller parameters & $\times$: non-adaptive & $\checkmark$: easy to implement \\
    \hline
     PLL resynthesis \cite{Blaabjerg_DoublePLL,Iravani_FrequencySynch} & $\times$: no theoretical guarantees & $\times$: no rigorous stability guarantees & $\checkmark$: designed around nominal condition & $\times$: major alteration in PLL & $\times$: non-adaptive & $\checkmark$: reasonably easy to implement \\
    \hline 
    Switch between GFL and GFM \cite{liu2024adaptive} & $\checkmark$: rigorous empirical evidence & $\times$: no rigorous stability guarantees  & $\checkmark$: retains core controls & $\times$: requires additional GFM controls & $\checkmark$: adapts to off-nominal conditions & $\times$: highly complex \\
    \hline 
    Outer loop resynthesis \cite{Oriol_IBRWeakGrid} & $\checkmark$: locally robust controllers & $\times$: no rigorous stability guarantees & $\checkmark$: uses gain scheduling & $\times$: major changes involved & $\checkmark$: adapts to off-nominal conditions & $\times$: highly complex\\
    \hline 
    Controller retuning and supplementary control  \cite{mittal2024stability,azimi2022supplementary,li2024quasi} & $\times$: no theoretical guarantees & $\times$: no rigorous stability guarantees & $\checkmark$: retains core design & $\checkmark$: supplementary addition & $\times$: non-adaptive & $\checkmark$: reasonably easy to implement \\
    \hline 
    Proposed robust adaptive controller & $\checkmark$: theoretically guaranteed & $\checkmark$: analytical stability certificate & $\checkmark$: retains core design & $\checkmark$: only supplementary addition & $\checkmark$: adapts to off-nominal conditions & $\checkmark$: easy to implement \\
    \hline
\end{tabular}}
\end{table*}

\underline{(ii.D)}: \sina{Reference~\cite{azimi2022supplementary} designs a passivity-based supplementary controller for IBRs in weak grids to damp oscillations in output power and frequency. The proposed supplementary controller modulates voltage signal generated by the PI-current controller. 
%Reference~\cite{azimi2021supplementary} proposes a supplementary controller represented in port-controlled Hamiltonian framework for inverter-based distributed generations to damp oscillations in the power and frequency. 
Reference~\cite{li2024quasi} develops a quasi-harmonic voltage supplementary control to damp weak-grid SSOs. The proposed supplementary control takes into account the \sina{pulse-width modulation} \sina{(}PWM\sina{)} delay and is injected to the modulation signal to suppress SSO in the POI voltage.}
% \begin{table*}[h!]
% \centering
% \caption{Pros and Cons of Control solutions\label{Tab:Methods}}
% \begin{tabular}{ |p{5cm}||p{1cm}|p{3cm}|p{4cm}|p{2.5cm}|   }
%  \hline
%  %\multicolumn{5}{|c|} \\
%  %\hline
%  \rowcolor{gray!10}
%  Attributes &Robust&Regirous stability analysis&Not affecting nominal performance&Minimally invasive\\
%  \hline
%  \hline
%  Standard controller retuning  &\centering$\times$&\centering$\times$&\centering$\times$&\centerline{\checkmark}
%  \\
%  \hline
%  PLL redesign&\centering$\times$&\centering$\times$&\centering\checkmark&\centerline{$\times$}\\
%  \hline
%  Switch between GFL to GFM~\cite{liu2024adaptive}&\centering$\times$&\centering$\times$&\centering\checkmark&\centerline{$\times$}\\
%  \hline
%  Proposed controller &\centering\checkmark&\centering\checkmark&\centering\checkmark&\centerline{\checkmark}\\
%  \hline
%  Linear control and outer loop changes&\centering$\times$&\centering$\times$&\centering\checkmark&\centerline{$\times$}\\
%  \hline
% \end{tabular}
% \end{table*}

\underline{(ii.E)}: Examples of resynthesis philosophy include double-PLL \cite{Blaabjerg_DoublePLL} and frequency-based synchronization control augmenting the existing PLL \cite{Iravani_FrequencySynch}. Reference \cite{Oriol_IBRWeakGrid} proposes altering outer-loop structure by introducing real power ($P$)-voltage ($V$) decoupling controls that are designed using $35$ local robust controllers. A gain scheduling approach is used to smoothly change controller dynamics with operating point changes. A simpler controller with the $P$-$V$ decoupling philosophy was proposed in \cite{LingLing_WindWeakGrid}. \nrc{In~\cite{liu2024adaptive}, a nonlinear controller is implemented using a hysteresis loop-like control to smoothly switch between grid-following mode to the grid-forming mode when the system transitions from a strong grid to a weak grid.}

\subsection{Gaps in Literature}
Table~\ref{Tab:Methods} summarizes the literature on SSO damping using IBR-focused measures and highlights the gaps/limitations in light of \nrc{six} attributes, \nrc{four of} which are elaborated below. \nrc{The remaining two are self-explanatory from the table.} Note that grid strength-focused measures may not be adequate alone and are part of medium to long-term planning time horizon if new infrastructure needs to be developed. Therefore, going forward, this philosophy will not be discussed. 

\textit{1. Robustness to uncertainty and disturbance:} System operating condition influenced by factors including load flow and system topology, which are uncertain, has a significant impact on the SSO mode. In addition, there are unmeasureable disturbances within the system that can impact SSO damping performance. Controller retuning performed in \cite{Huang-12-ERCOT,E3C19,NERCguide17,Liu-17-SSCI,IBRSSOTF,JunLiang_VSI_WeakGrid,Lennart_PLL_WeakGrid}, retuning followed by supplementary controls proposed in \cite{mittal2024stability}, and control resynthesis in \cite{Blaabjerg_DoublePLL,Iravani_FrequencySynch,LingLing_WindWeakGrid} are based upon design around nominal operating points using linear control theory and provides \textit{only limited robustness and performance guarantees}. Adaptive control approaches proposed in \cite{Oriol_IBRWeakGrid} and \cite{liu2024adaptive} on the other hand considers robustness against change in operating conditions, but do suffer from certain drawbacks elaborated next. \nrc{The supplementary controller proposed in reference \cite{azimi2022supplementary} mitigates disturbances in the system that are functions of state variables, and derivatives of voltages and currents, all of which are assumed to be measurable. Moreover, the authors did not consider parametric uncertainty in their system. Reference \cite{li2024quasi} on the other hand, does not consider any disturbance and parametric uncertainty during control design. }

\textit{2. Scientific rigor in the form of analytical guarantees:}
Although numerical stability analysis was presented, virtually none of the papers in literature of SSO damping \cite{Huang-12-ERCOT,E3C19,NERCguide17,Liu-17-SSCI,IBRSSOTF,JunLiang_VSI_WeakGrid,Lennart_PLL_WeakGrid,mittal2024stability,Blaabjerg_DoublePLL,Iravani_FrequencySynch,Oriol_IBRWeakGrid,LingLing_WindWeakGrid,liu2024adaptive} provided analytical guarantees of stability and performance facing uncertainty in operating condition and disturbances. The gain-scheduled adaptive controller proposed in \cite{Oriol_IBRWeakGrid} changes the controller based on IBR power output with \sina{no theoretical stability/performance guarantees}. \nrc{The hysteresis-type adaptive control proposed in \cite{liu2024adaptive} uses a heuristic to create a map between SCR and the ratio constants needed for GFL-GFM mode change without solid justification. %providing a guaranteed theoretical stability analysis. 
Moreover, no rigorous analytical guarantees were provided for stability and performance in either of \cite{liu2024adaptive} and  \cite{Oriol_IBRWeakGrid}.} \nrc{Reference \cite{li2024quasi} does not provide any rigorous analytical stability guarantees in presence of the proposed supplementary control.}

\textit{3. Impact on nominal performance:} Standard GFLC controls involving hierarchical controls are designed under certain nominal conditions in which the SCR might not be low. Except adaptive controllers proposed in \cite{Oriol_IBRWeakGrid,liu2024adaptive}, other approaches \cite{Huang-12-ERCOT,E3C19,NERCguide17,Liu-17-SSCI,IBRSSOTF,mittal2024stability,Blaabjerg_DoublePLL,Iravani_FrequencySynch} will lead to certain conservatism, i.e., they will need to sacrifice performance under nominal condition while being able to avoid SSO under weak grid conditions.

\textit{4. Change required in the standard GFLC controls in existing IBRs:} Controller retuning proposed in \cite{Huang-12-ERCOT,E3C19,NERCguide17,Liu-17-SSCI,IBRSSOTF} is advantageous in this regard. However, other measures will face significant challenges. \nrc{Controller resynthesis proposed in \cite{Blaabjerg_DoublePLL,Iravani_FrequencySynch,Oriol_IBRWeakGrid,liu2024adaptive} demand that at least parts of the existing standard GFLC control architecture need to change, which is difficult to do for existing IBRs in the field.} Similarly, it can be very challenging to determine the best possible combination of control modes in individual IBRs as proposed in \cite{mittal2024stability}, since there are many existing IBRs in the field as opposed to only two IBRs considered therein.  

Finally, to the best of our knowledge, except \cite{azimi2022supplementary}, all of the existing papers have validated the proposed controls when the rest of the grid is modeled using a voltage source behind constant impedance as opposed to a multimachine power system with synchronous generators (SGs).

\subsection{Contributions of This Work}
%In view of the gaps in literature, we propose a \textit{robust adaptive supplementary} SSO damping controller. 
\nrc{Although oscillation damping control is not fundamentally new in power systems literature, the SSO phenomena is a new problem. The contributions of this paper in the particular context of SSOs are the following} %with the following attributes 
--\\
% (a) It avoids complicated retuning and resysnthesis effort of legacy controls, thereby avoiding any compromise on the nominal performance.\\
% (b) It adds supplementary modulating control inputs that act as an \textit{add-on} feature on legacy controls, which will avoid expensive alteration of GFLCs already in the field and reduce design cycle for new GFLCs.\\
\nrc{(a) Nonlinear robust adaptive control is proposed for the first time for damping IBR-induced SSOs.}\\
\nrc{(b) As opposed to papers in SSO damping literature, the proposed controller provides theoretically rigorous guarantees of} %The controller is 
robustness to uncertainties in the operating conditions and system disturbances.\\
\nrc{(c) Unlike most SSO damping papers, analytical certificates of stability and damping performance are derived during the controller design process.}\\ %The control design is based on rigorous stability, robustness, and damping performance analysis.\\
%\textcolor{black}{(e) 
\nrc{(d) As opposed to existing literature, we consider two multimachine systems with SGs (including the $16$-machine IEEE system) for validating the effectiveness of the proposed control, in addition to a single-GFLC-infinite-bus system. }

Models based ons space phasor calculus (SPC), quasistationary phasor calculus (QPC), and EMT frameworks are used for simulations. 

\cl{\textcolor{blue}{The approach presented has a few advantages over other possible nonlinear control strategies. First, the control law developed has a simple form and can handle uncertainty. Furthermore, stability is proven and bounds on performance are provided. Although one could use, for example, sliding mode control to address uncertainty, this would be done at the cost of having high gains and high frequency chattering effects. Other nonlinear control approaches, such as backstepping, require a specific ``triangular'' system structure and often result in complex control algorithms. In summary, in our view, the approach proposed allows one to design a relatively simple control law that can handle parametric uncertainty and mitigates disturbance effects.}}
% The main contribution of the paper is as follows:
% \begin{enumerate}
%     \item Develop a robust adaptive controller for the terminal voltage of a GFLC to mitigate with SSO problem caused by a weak grid condition 
%     \item Rigorous stability and robustness analysis based on the dissipativity theory
% \end{enumerate}

\section{Notations and Preambles}
The following notations and conventions are employed throughout the paper:
$\mathbb{R}$, $\mathbb{R}^n$, $\mathbb{R}^{n\times m}$ denote the space of real numbers, real vectors of length $n$ and real matrices of $n$ rows and $m$ columns, respectively; $\mathbb{R}_{+}$ denotes the set of non-negative real numbers, and
$\mathbb{R}_{++}$ denotes the set of strictly positive real numbers; and 
$X^\top$ denotes the transpose of the quantity $X$.
Normal-face lower-case letters (e.g., $x\in\mathbb{R}$) are used to represent real scalars, bold-face lower-case letters (e.g., $\mathbf{x}\in\mathbb{R}^n$) represent vectors, normal-face upper case letters (e.g., $X\in\mathbb{R}^{n\times m}$) represent matrices, while $X\succ0 ~(\succeq0)$ denotes a positive definite (semi-definite) matrix. The standard Euclidean norm for the vector $\mathbf{x}$ is denoted by $\left\|\mathbf{x}\right\|$.
The space of all square-integrable functions is defined as:
\begin{align*}
\mathcal{L}_2^n\triangleq\left\{\textbf{f}:\mathbb{R}_{+}\rightarrow\mathbb{R}^n\left|\sqrt{ \int_0^\infty \left\|\mathbf{f}(\tau)\right\|^2d\tau}<\infty\right.\right\}.
\end{align*}
For each $T\in\mathbb{R}_{+}$, the function $\textbf{f}_T:\mathbb{R}_{+}\rightarrow\mathbb{R}^n$, given by
\begin{align*}
\textbf{f}_T(t)\triangleq\left\{\begin{array}{rl}\textbf{f}(t),&0\le t\le T\\\textbf{0},&t> T\end{array}\right.,
\end{align*}
is called the \emph{truncation} of $\textbf{f}$ on the interval $[0,\hspace{2mm}T]$. Consequently, the set $\mathcal{L}_{2e}^n$ of all measurable signal $\textbf{f}:\mathbb{R}_{+}\rightarrow\mathbb{R}^n$ such that $\textbf{f}_T(t)\in\mathcal{L}_2$ for all $T\in[0,\hspace{2mm}\infty)$ is called an extension of $\mathcal{L}_2^n$ or the extended $\mathcal{L}_2$-space. 
\begin{definition}(Finite-Gain $\mathcal{L}$-stability)\cite{khalil2002nonlinear}
\noindent Consider the nonlinear system
\begin{align} \label{eqn:General}
   \mathcal{H}:\hspace{5mm} \begin{array}{rl}
    \dot{\mathbf{x}}&=\mathbf{f}(\mathbf{x},\mathbf{w})\\
          \mathbf{z}&=\mathbf{h}(\mathbf{x})
    \end{array}
\end{align}
where $\mathbf{x}\in\mathcal{L}_{2e}^n$, $\mathbf{w}\in\mathcal{L}_{2e}^p$, $\mathbf{z}\in\mathcal{L}_{2e}^m$ are the state, input, and output vector signals, respectively. The system in \eqref{eqn:General}, considered as a mapping of the form $\mathcal{H}:\mathcal{L}^p_{2e}\rightarrow\mathcal{L}^m_{2e}$ is said to be finite-gain $\mathcal{L}_2$-stable if there exists real non-negative constants $\gamma,~\varphi$ such that $ \left\|\mathcal{H}(\mathbf{w})\right\|\leq \gamma\left\|\mathbf{w}\right\|+\varphi$. Moreover, $\gamma^*=\inf\left\{\gamma\hspace{2mm}:\hspace{1mm} \left\|\mathcal{H}(\mathbf{w})\right\|\leq \gamma\left\|\mathbf{w}\right\|+\varphi \right\}
$ is the corresponding $\mathcal{L}_2$-gain (will be referred to as gain) of the system.
\end{definition}

\begin{definition}($\gamma$-dissipativity)\cite{van2000l2}
The dynamic system \eqref{eqn:General} is said to be dissipative with respect to the supply rate $s(\mathbf{w},\mathbf{z})\in \mathbb{R}$, if there exists an energy function $V(\mathbf{x})\geq 0$ such that, for all $t_f\geq t_0$, 
\begin{align}\label{eqn:dissipativity ineq}
    V(\mathbf{x}(t_f))\leq V(\mathbf{x}(t_0))+\int_{t_0}^{t_f} s\left(\mathbf{w}(\tau),\mathbf{z}(\tau)\right) d\tau\hspace{2mm}\forall\hspace{2mm} \mathbf{w}\in\mathcal{L}^p_{2e}.
\end{align}
Moreover, given a positive scalar $\gamma$, if the supply rate is taken as $s(\mathbf{w},\mathbf{z})=\gamma^2\left\|\mathbf{w}\right\|^2-\left\|\mathbf{z}\right\|^2$,
then the dissipation inequality in \eqref{eqn:dissipativity ineq} implies a finite-gain $\mathcal{L}_2$ stability and the system is said to be $\gamma$-dissipative. Consequently, the dissipativity inequality in \eqref{eqn:dissipativity ineq} becomes
\begin{align*}
    \dot{V}\leq \gamma^2\left\|\mathbf{w}\right\|^2-\left\|\mathbf{z}\right\|^2.
\end{align*}
\end{definition}
\section{Modeling of Grids with SGs and GFLCs} \label{sec:modeling}
This section presents a brief description of mathematical representation used for modeling the power system under consideration. Dynamic models of SGs and GFLCs are developed in respective rotating $d$-$q$ frames, while the transmission network and loads are represented in a synchronously rotating $D$-$Q$ frame. The models adopt SPC framework, which is built on space phasors (also called space-vectors) denoted by $\bar{x}(t) \in \mathcal{C}$. 
The time-varying phasor $\bar{x}_{dq}(t)$ is calculated by the frequency shift operation on $\bar{x}(t)$ leading to $\bar{x}_{dq}(t) = x_d(t) + jx_q(t) = \bar{x}(t)e^{-j\rho(t)}$, where $\frac{\mathrm{d} \rho (t) }{\mathrm{d} t} = \omega (t)$ is the angular speed of the rotating $dq$ frame. The operator $\bar{\Upsilon} : \mathcal{M} \rightarrow \mathcal{D}$ is defined as $\bar{x}_{bb}(t) = \bar{\Upsilon}(\textbf{x}(t))$, where $\mathcal{M}$ is a vector space representing the set of balanced three-phase signals. It can be shown that 
$\bar{\Upsilon}(\textbf{x}(t)) = [1~j~0] \mathcal{P}(t)\textbf{x}(t)$, where $\mathcal{P}(t)$ 
is the Park's transformation matrix.
%and are utilized in Section~\ref{sec:results} to validate the proposed control solutions. 
Readers are referred to \cite{lilan_2024modeling} for a more comprehensive treatment.

\subsection{Modeling of GFLC and its Controls}\label{sec:GFLC}
A generic circuit diagram of a GFLC is illustrated in Fig.~\ref{fig:circuit_gflc}. The dc-side of the converter represents a functional model of a renewable energy source where $c_c$ and $\tau_c$ are the dc-link capacitance and a delay in current source, respectively. Note that converter losses are added in the form of resistance $r_{on}$ to the filter resistance.  

\begin{figure}[h]
    \centerline{
\includegraphics[width=0.42\textwidth]{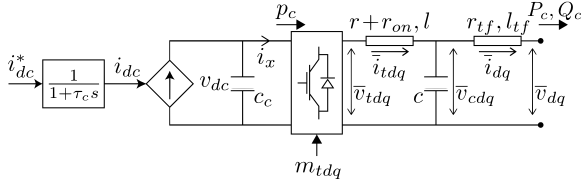}}
    \vspace{-5pt}
    \caption{Circuit model of GFLC [parameters: $\tau_{c}$ = $0.05$ s, $c_{c}$ = $1.7370$ pu, $r$ = $0.0033$ pu, $r_{on}$ = $0.0023$ pu, $l$ = $0.2454$ pu, $r_{t}$ = $0$ pu, $l_{t}$ = $0.1500$ pu, $s_{base}$ = $100$ MVA, $v_{dc,base}$ = $48.9873$ kV, $v_{ac,base}$ = $20$ kV].}
    \label{fig:circuit_gflc}
\end{figure}

The ac-side of the converter is modeled and controlled in a rotating $d$-$q$ frame generated by the phase-locked-loop (PLL) (see Fig.~\ref{fig:pll_control}), which is used to align the $d$-axis of the rotating frame along the voltage vector of the grid ($\overline{v}_{DQ}$), maintaining a zero $q$-axis voltage. The PLL model derives the reference angle $\delta_{pll}$ from the angular frequency deviation $\Delta\omega_{pll}$, which in turn is used for coordinate transformation between $DQ$ and $dq$ frames. The ac-side filter connects the converter's ac terminal to the grid through a transformer and typically consists of $r$-$l$-$c$ elements. However, we neglect the capacitance $c$ that is small and is used to filter out switching harmonics. Thus, the series impedance $(r+r_{on})$-$l$ is combined with series resistance and leakage inductance, $r_{tf}$ and $l_{tf}$ of the transformer when modeling and that leads to $\overline{i}_{tdq}$ = $\overline{i}_{dq}$.

% \begin{align}\label{eqn:optimization prob.}
%     minimize \hspace{20pt} & \hspace{-2pt} -\underline{x}+\mu k \\ \nonumber
%     {subject~to} \hspace{15pt} & k>0,\hspace{1mm}
% P \succ 0,\hspace{1mm}
%       \underline{Q}-\alpha P \succ 0,\hspace{1mm}
%      \overline{Q}-\alpha P \succ 0,\\\nonumber&\hspace{1mm}
%     \underline{x}\geq1,\hspace{1mm}
%     x_v\geq \underline{x},\hspace{1mm}x_{\eta}\geq\underline{x},\hspace{1mm}
%       x_i\geq \underline{x}.
% \end{align}

\begin{figure}[h!]
    \centerline{
    \includegraphics[width=0.4\textwidth]{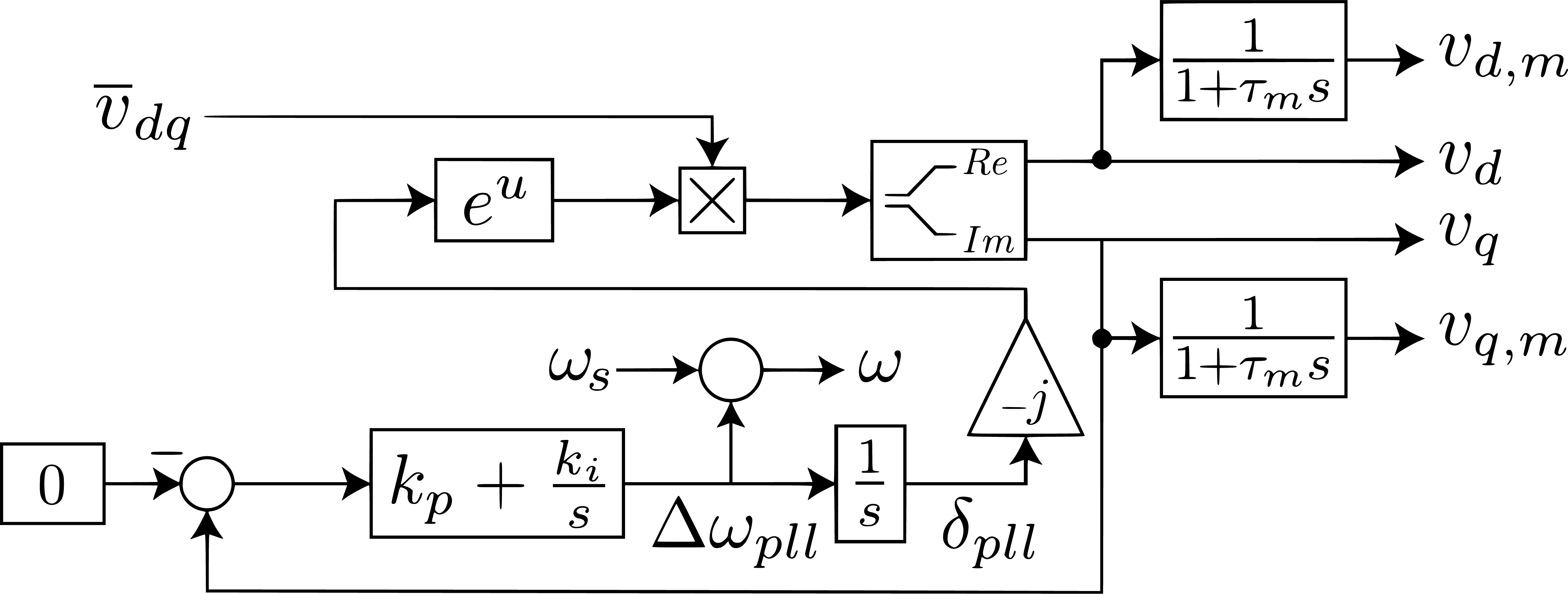}}
    \vspace{-5pt}
    \caption{Phase-locked loop (PLL) [parameters: $k_{p}$ = $101$, $k_{i}$ = $2562$ for $20$ Hz bandwidth \cite{pll_impact}, $\tau_{m}$ = $1$ ms].}
    \label{fig:pll_control}
\end{figure}

\begin{figure}[h!]
    \centerline{
    \includegraphics[width=0.5\textwidth]{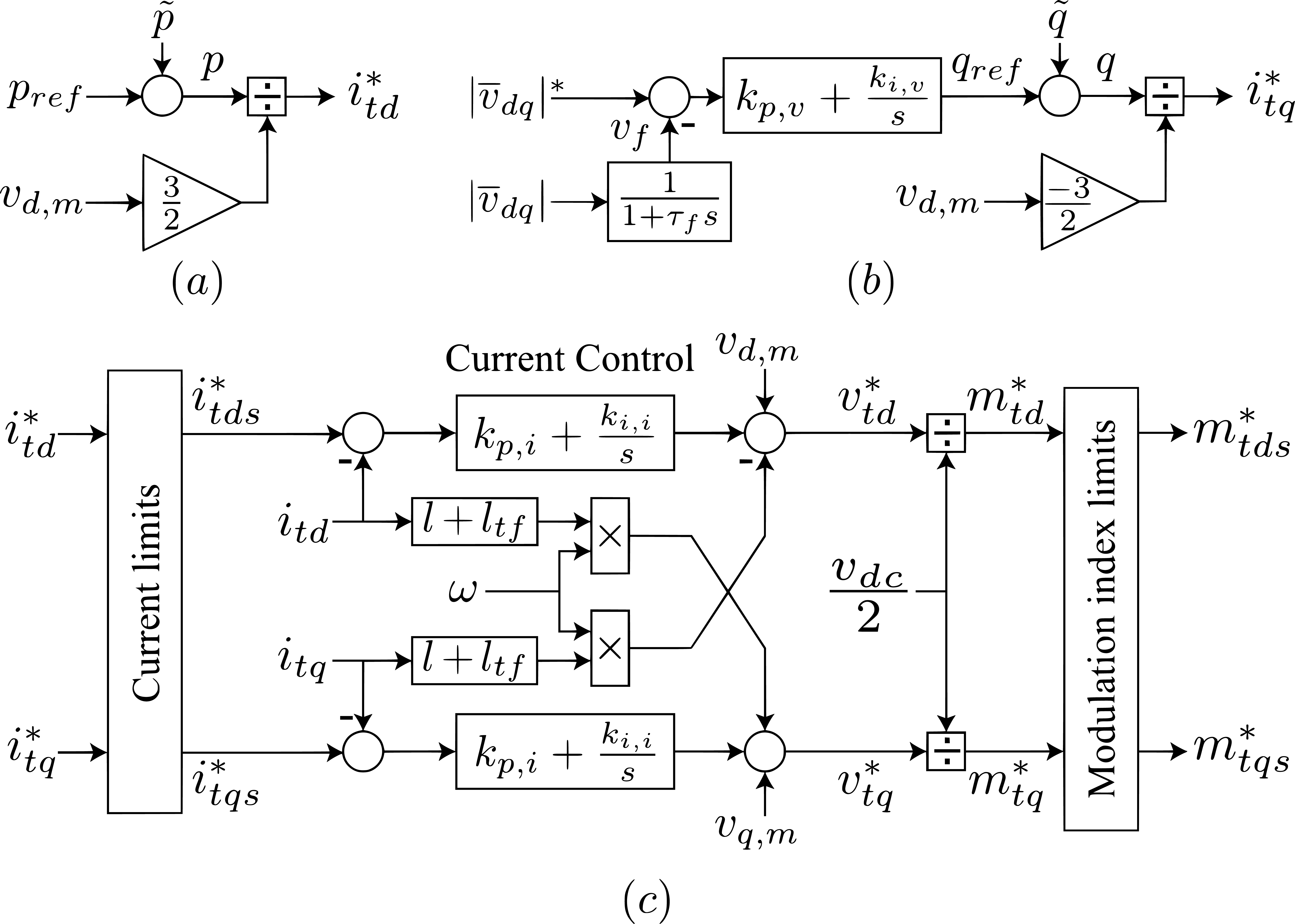}}
    \vspace{-5pt}
    \caption{\nrc{(a) Active power control, (b) outer voltage control, and (c) inner current control. [parameters: $p_{ref}$ = $7.00$ pu, $\tau_{f}$ = $50$ ms, $s_{base}$ = $100$ MVA, $v_{ac,base}$ = $20$ kV]. Supplementary control inputs $\Tilde{p}$ and $\Tilde{q}$ are used for SSO damping.}}
    \label{fig:gflc_control}
\end{figure}

Current references $i_d^*$ and $i_q^*$ are generated by the active power control and outer voltage control loop, respectively, as depicted in Fig.~\ref{fig:gflc_control}(a) and (b). A PI controller, that regulates the voltage magnitude $|\overline{v}_{dq}|$ at the point of common coupling, generates the respective reactive power reference \sina{$q_{ref}$} for the converter\sina{, which is later modulated by the supplementary reactive power control signal $\widetilde{q}$}. \sina{Similarly, the reference active power $p_{ref}$ is modulated by the supplementary active power control signal $\widetilde{p}$.} The ac current limits are imposed to protect the converter under short circuit faults in the grid considering a $15$\% overloading capacity. Standard inner current control loops shown in Fig.~\ref{fig:gflc_control}(c) are employed to regulate current through the series impedance of the filter and the transformer. Moreover, they generate the modulation indices $m_{td}^*$ and $m_{tq}^*$ for the switching circuit, which in-turn are determined by the ac terminal voltage references $v_{td}^*$ and $v_{tq}^*$, respectively. The magnitude of the modulation index is limited to $1$. The supplementary control inputs $\widetilde{p}$ and $\widetilde{q}$ are used for SSO damping, which are produced by a robust adaptive controller as described in Section~\ref{sec:ControlDevelop}.

\subsection{Modeling of SGs, Transmission Network, and Loads}
\nrc{SPC-based model:} The SGs consider $8$th-order subtransient models in SPC framework with stator transients in their respective $d$-$q$ frames rotating at corresponding rotor speeds. Additionally, the models include associated turbines, governors, and exciters. The transmission network is represented using a SPC-based dynamic model of the lumped $\pi$-section in a synchronously rotating $D$-$Q$ reference frame. Constant impedance loads are represented in the same $DQ$ frame by dynamic models consisting of parallel $r$-$l$-$c$ elements at the respective load buses.

\nrc{QPC-based model: For larger systems, a QPC-based model is used that considers a $6$th-order SG model and an algebraic transmission line and load models.}

% \section{Control Development}\label{sec:ControlDevelop}
\section{Robust Adaptive Supplementary SSO Damping Control}\label{sec:ControlDevelop}
As explained in the Introduction, damping weak-grid SSOs involving GFLCs is the objective of this paper. To this end, a robust adaptive SSO damping control development is presented in this section, which is divided into the following steps -- (A) Problem Formulation and (B) Controller Design.  

\subsection{Problem Formulation}\label{sec:ProbFormulation}
In the first step of problem formulation, our goal is to mathematically model the uncertainties in system operating condition affected by factors including load flow and system topology, in addition to unmeasurable disturbances within the system that can impact SSO modes and damping performance. As shown in Fig.~\ref{fig:UncertaintyMdl}, we represent the grid with uncertainties using a Thevenin's equivalent circuit connected to POI of the GFLC (bus $2$, in this case). The uncertainties come from unknown but bounded parameters $r_g\in\mathbb{R}_{++}$, $l_g\in\mathbb{R}_{++}$, and $c_g\in\mathbb{R}_{++}$, in addition to unknown Thevenin's voltage $\bar{v}_{gdq}\in \mathcal{C}$. Finally, two important aspects are worth noting:\\
1. The uncertainties in series resistance $r_g$, series inductance $l_g$, and the disturbance $\bar{v}_{gdq}$ can be lumped into the current $\bar{i}_{gdq}\in \mathcal{C}$. \nrc{Since the locations of the $d$ and the $q$ axes are determined by the PLL dynamics, its impact will be reflected in $\bar{i}_{gdq}$.}\\
2. The uncertainty in shunt capacitance $c_g$ needs to be considered separately.\\

\begin{figure}[h!]
    \centerline{    \includegraphics[width=0.30\textwidth]{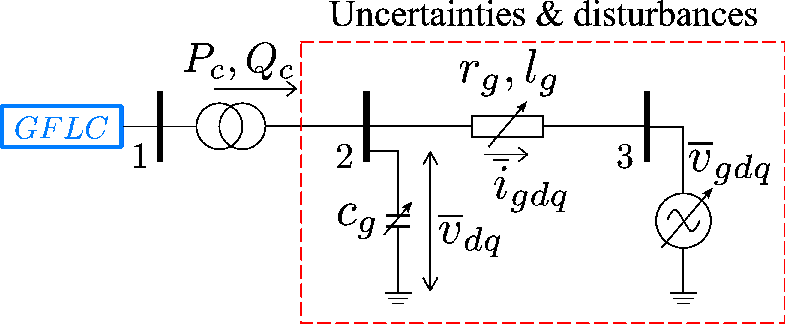}}
    \vspace{-5pt}
    \caption{Model of uncertainty in system parameters and disturbances of the equivalent grid model connected to the GFLC.}
    \label{fig:UncertaintyMdl}
\end{figure}

%Filter's Dynamics in d-q coordinate is as follows:
%\begin{align}\label{eqn:filter.dyn}
%    \left\{\begin{array}{rl}\dot{v}_{fd}&=-\sina{\beta} v_{fd}+\sina{\beta}v_d\\
%    \dot{v}_{fq}&=-\sina{\beta}v_{fq}+\sina{\beta}v_{q}
%    \end{array}\right.,
%\end{align}
%\begin{align}\label{eqn:filter.dyn}
        %\dot{\mathbf{v}}_f
%=-\sina{\beta}\mathbf{v}_f+\sina{\beta}\mathbf{v},
%\end{align}
%where
%$v_{fd}\in\mathbb{R}$, and $v_{fq}\in\mathbb{R}$ are the d, and q components of the filtered terminal voltage, respectively. $v_d\in\mathbb{R}$,and $v_q\in\mathbb{R}$  are the d and q components of the terminal voltage, respectively, $\sina{\beta}\in\mathbb{R}_{++}$ is a known constant. 
The proposed uncertainty model has the capability of considering a range of load flow conditions including different power factors, change in topology and grid strengths, and unknown disturbances within the grid. Mathematically, the POI voltage dynamics in the $dq$ coordinates can be expressed as follows
% \nrc{NRC: $\omega$ should be changed to $\omega_s$, $i_{gd}$ should be changed to $i_{gD}$}
\begin{align}\label{eqn: voltage.dyn}
    \left\{\begin{array}{rl}
 c_g\dot{v}_d&=c_g\omega v_q+\frac{2}{3v_d}p-i_{gd}\\
 c_g\dot{v}_q&=-c_g\omega v_d-\frac{2}{3v_d}q-i_{gq}\end{array}\right.,
%c_g\dot{\mathbf{v}}=-\omega c_gJ\mathbf{v}+\mathbf{i}-\mathbf{i}_g   %c_g\dot{v}_d&=\omega c_g v_q+i_d-i_{gd},\\
  %  c_g\dot{v}_q&=-\omega c_g v_d+i_q-i_{gq},
\end{align}
where $v_d\in\mathbb{R}$, and $v_q\in\mathbb{R}$, are respectively, the $d$- and $q$-component of the POI voltage, $\omega\in\mathbb{R}_{++}$ is the PLL frequency, which is measurable, $c_g$ is unknown but bounded such that $c_g\in\begin{bmatrix}
   \underbar{c}_{g},\overline{c}_g
\end{bmatrix}$,  where $\underbar{c}_g\in\mathbb{R}_{++}$, and $\overline{c}_g\in\mathbb{R}_{++}$ are known constants. Moreover, $i_{gd}\in\mathcal{L}_{2e}$, and $i_{gq}\in\mathcal{L}_{2e}$  are the $d$- and $q$- components of the grid current disturbance, \nrc{$q\in\mathbb{R}$, and $p\in\mathbb{R}$ are the reactive and active power control inputs to be designed as shown in Fig.~\ref{fig:gflc_control}. Note that $Q_c$ and $P_c$ are the actual reactive and real power outputs of the GFLC at the POI as shown in Figs~\ref{fig:circuit_gflc} and \ref{fig:UncertaintyMdl} that are different from $q$ and $p$. Equation \eqref{eqn: voltage.dyn} assumes $v_d$ $\approx$ $v_{d, m}$, since $\tau_m$ is $1$ ms, which is very small.  Also, \eqref{eqn: voltage.dyn} assumes that the inner current control loops shown in Fig.~\ref{fig:gflc_control}(c) are significantly faster compared to the outer loops, so that an ideal tracking of current references can be assumed.} Moreover, we do not take the limits on the current references in our problem formulation, which will be left for future work. 
%Similarly, one can define the terminal voltage dynamics in the D-Q coordinate as follows
%\begin{align}\label{eqn: voltage.dyn-DQ}
%     \left\{\begin{array}{rl}
 %c_g\dot{v}_D&=\omega_s c_gv_Q+i_D-i_{gD}\\
 %c_g\dot{v}_Q&=-\omega_s c_gv_D+i_Q-i_{gQ}\end{array}\right.,
%\end{align}
%where $\omega_s\in\mathbb{R}_{++}$ is the fundamental grid frequency.

%\begin{assumption}\label{assump:PLL}
%    The PLL dynamics is faster than the filter dynamics such that $v_q=0$.
%\end{assumption}

In the second step of the problem formulation, we describe the overall control objectives as the following: \textit{Find a control law along with an adaptation law such that the SSOs are damped within a certain time frame without retuning or resynthesis of legacy controls outlined in \ref{sec:GFLC} while being robust to disturbance and uncertainties in \eqref{eqn: voltage.dyn}.}

% while tracking the desired terminal voltage. The sub-synchronous oscillation in this paper is triggered via disturbance injection or the system parameter variation.  

% \subsection{Control Objectives}\label{sec:Objectives}
% The main control objective is to damp subsynchronous oscillations while tracking the terminal voltage at the operating point via developing a supplementary adaptive control mechanism added to a PI controller  depicted in Fig. The underlying potential reason for subsynchronous oscillation is the filter dynamics in~\eqref{eqn:filter.dyn} so that if the inverter is subjected to a disturbance or parametric uncertainty it generates the oscillations. In this paper, the disturbance occurs due to grid current variations, and parametric uncertainty exists due to an unknown capacitor at the 
%  inverter terminal. 
 
The overall control objective is divided into several control sub-objectives as follows 
\begin{enumerate}\label{ctrl.subobj}
 \item A supplementary SSO damping control philosophy needs to be adopted that modulates the real and reactive power reference signals $p_{ref}$ and $q_{ref}$ using supplementary control inputs $\widetilde{p}$ and $\widetilde{q}$, respectively, as shown in Fig.~\ref{fig:gflc_control}(a), (b). This avoids altering existing controllers.
 \item The SSOs should be damped within a predetermined settling time $\bar{\tau}$.
 \item Due to saturation limits in the inner current control loops (see, Fig.~\ref{fig:gflc_control}(c)), the control gain of the supplementary control should generate small variations in $\widetilde{p}$ and $\widetilde{q}$ in contrast to $p_{ref}$ and $q_{ref}$.
\item To make the controller robust against the grid current disturbance, the $\gamma$ upper bound on the $\mathcal{L}_2$ gain from the disturbance to the outputs should be as small as possible. However, this comes at the expense of increasing the control gain, and then the control signal. Hence, there is a trade off between having lower value for $\gamma$ and the control gain of the supplementary control.
\item The parametric uncertainty in the grid capacitance $c_g$ should be mitigated via an adaptation law. 
\end{enumerate}

\subsection{Controller Design}\label{sec:Design}
The filtered voltage dynamics is as follows
\begin{align}\label{eqn:filter.dyn}
    \dot{v}_f=-\sina{\beta}v_f+\sina{\beta}v_n,
\end{align}
where $v_n\triangleq\left\|\mathbf{v}\right\|$ with $\mathbf{v}=\begin{bmatrix}
    v_d\\
    v_q
\end{bmatrix}$, and $v_f\in\mathbb{R}$ is the magnitude of the filtered POI voltage (see Fig.~\ref{fig:gflc_control}(b)), and $\sina{\beta}=\frac{1}{\tau_f}\in\mathbb{R}_{++}$ with $\tau_f$ is the known time constant of the filter. 
% Note that $|\overline{v}_{DQ}|$ in Fig.~\ref{fig:gflc_control}(b) is equal to $v_{n}$ since the norm of voltage is independent of the reference frame.

  Denote the filtered POI voltage tracking error as follows
\begin{align}\label{eqn:filt.err}
     \widetilde{v}_f&\triangleq v_f-v_{n0},
\end{align}
where $v_{n0}\triangleq\left\|\mathbf{v}_0\right\| = |\overline{v}_{dq}|^*$ is the pre-disturbance nominal voltage with 
$\mathbf{v}_0=\begin{bmatrix}
    v_{d0}&
    v_{q0}
\end{bmatrix}^\top$. Similarly, define the POI voltage tracking error as follows
\begin{align}\label{eqn:volt.er}
     \widetilde{v}_n&\triangleq v_n-v_{n0}.
\end{align}
%Using Assumption~\ref{assump:PLL} the filter tracking error is defined as follows
%\begin{align}
%\label{eqn: filtererror}
%    \left\{\begin{array}{rl}
%\widetilde{v}_{fd}&\triangleq
%    0\\
%\widetilde{v}_{fq}&\triangleq v_{fd}-\left\|\mathbf{v}_0\right\|
%    \end{array}\right.,
%\end{align} 
%where %$\mathbf{v}_0=\begin{bmatrix}
%    v_{d0}\\
%    v_{q0}
%\end{bmatrix}$.
Denote the output of the PI controller in Fig.~\ref{fig:gflc_control}(b) with gains $k_{p,v},~k_{i,v} \in \mathbb{R}_{++}$  normalized w.r.t. $k_{p,v}$ as follows 
\begin{align}\label{eqn:er+inter}
\eta \triangleq \widetilde{v}_{f}+\psi\underbrace{\int_0^t \widetilde{v}_{f} (\tau)d\tau}_{\text{$\widetilde{v}_{fi}$}},
\end{align}
where $ \psi \triangleq \frac{k_{i,v}}{k_{p,v}} $, and $\eta = \frac{q_{ref} - q(0)}{k_{p,v}}$ with $q(0)$ as the pre-disturbance nominal value of $q$. Taking the first time-derivative of~\eqref{eqn:er+inter} yields
\begin{align}\label{eqn:etadot}
    \dot{\eta}=%\dot{\widetilde{v}}_{f}+\psi\widetilde{v}_{f} \nonumber \\
\dot{v}_f+\psi\widetilde{v}_f.
\end{align}
Inserting the filter dynamics from~\eqref{eqn:filter.dyn} in \eqref{eqn:etadot} results in
\begin{align} \label{eqn:etadot2}
    \dot{\eta}=-\sina{\beta}v_f+\sina{\beta}v_n+\psi\widetilde{v}_f.
\end{align}
Next, applying the errors defined in~\eqref{eqn:filt.err} and~\eqref{eqn:volt.er} in \eqref{eqn:etadot2} leads to
\begin{align*}
\dot{\eta}&=-\sina{\beta}\left(\widetilde{v}_f+v_{n0}\right)+\sina{\beta}\left(\widetilde{v}_n+v_{n0}\right)+\psi\widetilde{v}_f\\
&=\left(\psi-\sina{\beta}\right)\widetilde{v}_f+\sina{\beta}\widetilde{v}_n,
\end{align*}
then using~\eqref{eqn:er+inter} yields
\begin{align}\label{eqn:filt.track.dyn}
    \dot{\eta}=\left(\psi-\sina{\beta}\right)\left(\eta-\psi\widetilde{v}_{fi}\right)+\sina{\beta}\widetilde{v}_n.
\end{align}
Taking the time-derivative of the terminal voltage tracking error in~\eqref{eqn:volt.er}, we get
\begin{align*}
    \dot{\widetilde{v}}_n&=\dot{v}_n\\
    %&=\frac{d}{dt}\left(\left(v_d^2+v_q^2\right)^{\frac{1}{2}}\right)\\
    %&=\frac{1}{2}\left(2v_d\dot{v}_d+2v_q\dot{v}_q\right)\frac{1}{\sqrt{v_d^2+v_d^2}}\\
&=\frac{\left(v_d\dot{v}_d+v_q\dot{v}_q\right)}{v_n}.
\end{align*}
Next, using the terminal voltage dynamics in~\eqref{eqn: voltage.dyn} yields 
\begin{align}\label{eqn:nerr_dyn}
\dot{\widetilde{v}}_n&=\frac{v_d}{v_n}\left(\omega v_q+\frac{2}{3c_gv_d}p-\frac{i_{
gd}}{c_g}\right)\\\nonumber
&+\frac{v_q}{v_n}\left(-\omega v_d-\frac{2}{3c_gv_d}q-\frac{i_{gq}}{c_g}\right)\\\nonumber
&=\frac{2}{3c_g v_n}p-\frac{2v_q}{3c_gv_dv_n}q-\frac{v_d}{c_gv_n}i_{gd}-\frac{v_q}{c_g v_n}i_{gq}.
\end{align}
Define $\sina{\sigma_g}\triangleq\frac{1}{c_g}$, one can rewrite~\eqref{eqn:nerr_dyn} as follows
\begin{align}\label{eqn:openLoopSys.Dyn.}
    \dot{\widetilde{v}}_n=&\sina{\sigma_g}\left(\frac{2}{3\sina{v_n}}p-\frac{2 v_q}{3v_d\sina{v_n}}q\right)-\sina{\sigma_g}\left(\frac{v_d}{v_n}i_{gd}+\frac{ v_q}{v_n}i_{gq}\right).
\end{align}
Define
\begin{align}\label{eqn:w}
    w\triangleq & \frac{v_d}{v_n}i_{gd}+\frac{ v_q}{v_n}i_{gq},
\end{align}
\sina{
\begin{align}\label{eqn:u}
    u\triangleq u_p+u_q,
\end{align}
where 
\begin{align}\label{eqn:up}
u_p\triangleq\frac{2}{3v_n}p
\end{align}
\begin{align}\label{eqn:uq}
u_q\triangleq-\frac{2v_q}{3v_dv_n}q.
\end{align}
\begin{remark}
    Since the actual control inputs are $p$ and $q$, then the procedure of designing the controller is to design $u_p$, and $u_q$. %Then 
   Subsequently, we can obtain $p$, and $q$ using~\eqref{eqn:up},~\eqref{eqn:uq} that can be rewritten as follows
    \begin{align}\label{eqn:p}
    p = \frac{3}{2}v_n u_p,    
    \end{align}
    \begin{align}\label{eqn:q}
        q=-\frac{3v_d}{2v_q}v_nu_q.
    \end{align}
\end{remark}}
Then~\eqref{eqn:openLoopSys.Dyn.} can be rewritten as follows
\begin{align}\label{eqn:compact.open.loop.error.dyn}
    \dot{\widetilde{v}}_n=\sina{\sigma_g} u-\sina{\sigma_g}w.
\end{align}
Define $\widetilde{u}\triangleq u-u_0$, and $\widetilde{w}\triangleq w-w_0$, where $w_0=\frac{v_{d0}}{v_{n0}}i_{gd0}+\frac{v_{q0}}{v_{n0}}i_{gq0}$, $u_0=w_0$, then~\eqref{eqn:compact.open.loop.error.dyn} can be written as follows
\begin{align}\label{eqn:compact.open.loop.error2}
    \dot{\widetilde{v}}_n=\sina{\sigma_g}\widetilde{u}-\sina{\sigma_g}\widetilde{w}.
\end{align}
Define the lumped disturbance $\widetilde{w}_{\sina{\sigma}}\triangleq \sina{\sigma_g} \widetilde{w}$, then~\eqref{eqn:compact.open.loop.error2} is rewritten as follows
\begin{align}\label{eqn:finalOpenLoopDyn.}
    \dot{\widetilde{v}}_n=\sina{\sigma_g}\widetilde{u}-\widetilde{w}_{\sina{\sigma}}.
\end{align}
%Define $\widetilde{\theta}_g\triangleq \sina{\sigma_g}-\widehat{\theta}_g$, then inserting for $\sina{\sigma_g}$ yields
%\begin{align}\label{eqn:open-loop-error-dyn-final}
%\dot{\widetilde{v}}_n=\widetilde{\theta}_g\widetilde{u}+\widehat{\theta}_g\widetilde{u}-\widetilde{w}_{\theta},  
%\end{align}
Subsequently, the control law is designed as follows
\begin{align}\label{eqn:Cntrl.Law}
\widetilde{u}=\underbrace{\frac{-k}{\sina{\widehat{\sigma}_g}}\widetilde{v}_n}_{\sina{\text{$\widetilde{u}_p$}}}-k_{p,v}\eta,
\end{align}
where $\sina{\widehat{\sigma}_g}$ is the estimated parameter of \sina{$\sigma_{g}$} and $k\in\mathbb{R}_{++}$ is the gain of the supplementary control. The estimated parameter $\sina{\widehat{\sigma}_g}$ is calculated based on the adaptation law \eqref{eqn:adaptive law} derived later.

From the control designed in~\eqref{eqn:Cntrl.Law} and \sina{using~\eqref{eqn:p}}, the active power modulation signal shown in Figs~\ref{fig:gflc_control}(a), can be~\sina{designed as follows}
\begin{align}
        \widetilde{p}=&-\frac{3kv_n}{2\sina{\widehat{\sigma}_g}}\widetilde{v}_n
\end{align}
where $\widetilde{p}=p-p_{ref}$ is the active power modulation signal.
\nrc{\begin{remark}
In this paper, we consider the reactive power modulation signal~\sina{shown in Fig.3(b)}  $\widetilde{q}=0$. The underlying reason behind this is that from~\eqref{eqn:q}, $v_q$ will appear in the denominator, and when PLL achieves tracking, this would make the reactive power modulation signal $\widetilde{q}$ unbounded.
\end{remark}}
Substituting the control law developed in~\eqref{eqn:Cntrl.Law} into~\eqref{eqn:finalOpenLoopDyn.} yields 
\begin{align}\label{eqn:closed.loop.er.dyn}
    \dot{\widetilde{v}}_n=-k\sina{\frac{\sigma_g}{\widehat{\sigma}_g}}\widetilde{v}_n-\sina{\sigma_g}k_{p,v}\eta-\widetilde{w}_{\sina{\sigma}}.
\end{align}
Define \sina{parameter estimation error} $\sina{\widetilde{\sigma}_g}\triangleq \sina{\sigma_g}-\sina{\widehat{\sigma}_g}$, then the first term in~\eqref{eqn:closed.loop.er.dyn} can be written as follows  
\begin{align}\label{eqn:closed.loop.er.dyn_final}
    \dot{\widetilde{v}}_n&=-k\frac{\left(\sina{\widetilde{\sigma}_g}+\sina{\widehat{\sigma}_g}\right)}{\sina{\widehat{\sigma}_g}}\widetilde{v}_n-\sina{\sigma_g}k_{p,v}\eta-\widetilde{w}_{\sina{\sigma}}\\\nonumber
    &=-k\frac{\sina{\widetilde{\sigma}_g}}{\sina{\widehat{\sigma}_g}}\widetilde{v}_n-k\widetilde{v}_n-\sina{\sigma_g}k_{p,v}\eta-\widetilde{w}_{\sina{\sigma}}.
\end{align}
\\
Let
\begin{align}\label{eqn:matrix P}
    P_F\triangleq\begin{bmatrix}
        P&&\mathbf{0}\\
        \mathbf{0}&&1
    \end{bmatrix}\succ 0,\hspace{1mm} P&=\begin{bmatrix}
           p_v&& p_{v\eta}&& p_{vi}\\\star&& p_{\eta}&& p_{\eta i}\\ \star&&\star&& p_i
       \end{bmatrix} \succ 0,
          \end{align}
          \begin{align}\label{eqn:matrix Q}
 Q_F\triangleq\begin{bmatrix}
     Q&&\mathbf{0}\\
     \mathbf{0}&&0
 \end{bmatrix},\hspace{1mm}Q&=\begin{bmatrix}
           Q_v&&Q_{v\eta}&&Q_{vi}\\
           \star&&Q_{\eta}&&Q_{\eta i}\\
           \star&&\star&&Q_i
       \end{bmatrix}\succ 0,
\end{align}
   
 %\begin{align*}
  %  P&=\begin{bmatrix}
  %         p_v&& p_{v\eta}&& p_{vi}\\\star&& p_{\eta}&& p_{\eta i}\\ \star&&\star&& p_i
   %    \end{bmatrix},\\
%Q&=\begin{bmatrix}
           %Q_v&&Q_{v\eta}&&Q_{vi}\\
           %\star&&Q_{\eta}&&Q_{\eta i}\\
           %\star&&\star&&Q_i
       %\end{bmatrix}\succ 0,\hspace{1mm}\text{and}\hspace{1mm} \mathbf{0}\in\mathbb{R}^3,
       %\end{align*}
  with
    \begin{align}\label{eqn:Qv}
        Q_v&=p_vk-p_{v\eta}\sina{\beta}-1-\frac{1}{4\gamma_v^2},
        \end{align}
\begin{align}\label{eqn:Qeta}
       Q_{\eta}&=-p_{\eta}\left(\psi-\sina{\beta}\right)+p_{v\eta}k_{p,v}\sina{\sigma_g}-p_{\eta i}-1-\frac{1}{4\gamma_{\eta}^2},
        \end{align}
       \begin{align}\label{eqn:Qi}
  Q_i&=p_i\psi+p_{\eta i}\psi\left(\psi-\sina{\beta}\right)-1-\frac{1}{4\gamma_i^2},
  \end{align}
  \begin{align}\label{eqn:Qveta}
  Q_{v \eta}&=\frac{1}{2}\left(p_v k_{p,v}\sina{\sigma_g}-p_{\eta}\sina{\beta}-p_{v\eta}\left(\psi-\sina{\beta}\right)-p_{vi}+p_{v\eta}k\right),
  \end{align}
  \begin{align}\label{eqn:Qetai}
Q_{\eta i}&=\frac{1}{2}\left(p_{\eta}\psi\left(\psi-\sina{\beta}\right)+p_{vi}k_{p,v}\sina{\sigma_g}-p_i+p_{\eta i} \psi\right),
\end{align}
\begin{align}\label{eqn:Qvi}
  Q_{vi}&=\frac{1}{2}\left(p_{v\eta}\psi\left(\psi-\sina{\beta}\right)+p_{vi}\left(\psi+k\right)-p_{\eta i} \sina{\beta}\right),
    \end{align}
   for all $\sina{\sigma_g}\in\begin{bmatrix}
       \sina{\underline{\sigma}_g},\sina{\overline{\sigma}_g}
   \end{bmatrix}$. See Appendix for more details. In addition, $p_v\in\mathbb{R}$, $\sina{p}_{v\eta}\in\mathbb{R}$, $p_{vi}\in\mathbb{R}$, $p_{\eta}\in\mathbb{R}$, $p_{\eta 
 i}\in\mathbb{R}$, $p_i\in\mathbb{R}$. Moreover, $\gamma_v\in\mathbb{R}_{++}$, $\gamma_{i}\in\mathbb{R}_{++}$, and $\gamma_{\eta}\in\mathbb{R}_{++}$ are the the upper bound on the $\mathcal{L}_2$ gain from the lumped disturbance to the POI voltage norm error $\widetilde{v}_n$, integral of the filtered POI voltage error $\widetilde{v}_{fi}$, and the normalized PI controller output $\eta$, respectively. 
 
\sina{Since we are interested in minimizing $\gamma_v$, $\gamma_{\eta}$, and $\gamma_{i}$, one way to achieve this for all of them simultaneously is by minimizing the maximum of them defined by $\overline{\gamma}={max}\left\{\gamma_v,\gamma_{\eta},\gamma_{i}\right\}$. Next, to achieve the control sub-objectives, we need to formulate an optimization problem for damping SSOs and minimizing the upper bounds on the $\mathcal{L}_2$ gain of the closed-loop system from the disturbance to the outputs while minimizing the control gain $k$. To achieve the desired settling time $\Bar{\tau}$, we define $\alpha\triangleq\frac{1}{\Bar{\tau}} \in \mathbb{R}_{++}$ and then formulate it in the optimization problem. To perform a trade off between minimizing $\overline{\gamma}$ and the control gain $k$, we consider the weighting parameter $\mu$ in the objective functional. Since the terms $\frac{1}{\gamma_v^2}$, $\frac{1}{\gamma_{\eta}^2}$, and $\frac{1}{\gamma_i^2}$ in~\eqref{eqn:Qv},~\eqref{eqn:Qeta}, and~\eqref{eqn:Qi} are non-convex, we define $x_v\triangleq\frac{1}{\gamma_v^2}$, $x_{\eta}\triangleq\frac{1}{\gamma_{\eta}^2}$, $x_i\triangleq\frac{1}{\gamma_i^2}$, and to make it consistent with $\overline{\gamma}$, we define $\underline{x}\triangleq\frac{1}{\overline{\gamma}\sina{^2}}$. Then, to minimize $\overline{\gamma}$, we need to maximize $\underline{x}$ or minimizing $-\underline{x}$. Since we aim to attenuate the effect of the disturbance on the outputs as much as possible, we consider an additional constraint in the optimization problem such that $\overline{\gamma}\in(0,1]$. Consequently, the convex optimization problem can be set up as follows} 
\begin{align}\label{eqn:optimization prob.}
    minimize \hspace{20pt} & \hspace{-2pt} -\underline{x}+\mu k \\ \nonumber
    {subject~to} \hspace{15pt} & k>0,\hspace{1mm}
P \succ 0,\hspace{1mm}
      \underline{Q}-\alpha P \succ 0,\hspace{1mm}
     \overline{Q}-\alpha P \succ 0,\\\nonumber&\hspace{1mm}
    \underline{x}\geq1,\hspace{1mm}
    x_v\geq \underline{x},\hspace{1mm}x_{\eta}\geq\underline{x},\hspace{1mm}
      x_i\geq \underline{x}.
\end{align}

\nrc{If the system is robustly stabilizable with the proposed controller, then \eqref{eqn:optimization prob.} will have a solution if $\alpha$ is sufficiently small. However, the smaller the value of $\alpha$, the larger the time it takes for the oscillations to decay. Hence, a search should be done to obtain the largest value of $\alpha$ for which \eqref{eqn:optimization prob.} is feasible. As for $\mu$, several values should be considered to test different trade-offs between magnitude of the gain $k$ and mitigation of perturbations. For small values of $\mu$, one has strong mitigation of perturbation effects at the expense of having a gain $k$ of large magnitude. For large values of $\mu$, there is more emphasis on bounding the magnitude of the gain $k$ but this will come at the expense of having worse response to perturbations. So, different values of $\mu$ should be tested and the one with the best compromise should be chosen.}

Furthermore, $\underline{Q}$ and $\overline{Q}$ are obtained using $\sina{\underline{\sigma}_g}$ and $\sina{\overline{\sigma}_g}$, respectively.
Consequently, the adaptive law is designed as follows
\begin{align}\label{eqn:adaptive law}
    \sina{\dot{\widehat{\sigma}}_g}=-\frac{k\left(p_v\widetilde{v}_n+p_{v\eta}\eta+p_{vi}\widetilde{v}_{fi}
    \right)\widetilde{v}_n}{\sina{\widehat{\sigma}_g}},
\end{align}
where $p_v$, $p_{v\eta}$, and $p_{vi}$, are the adaptation gains obtained from solving the optimization problem~\eqref{eqn:optimization prob.}. 
\sina{The entire design procedure of the developed controller is summarized in a flowchart illustrated in Fig.~\ref{fig:control_proced}.}

\begin{figure}[h!]
    \centerline{\includegraphics[width=0.5\textwidth]{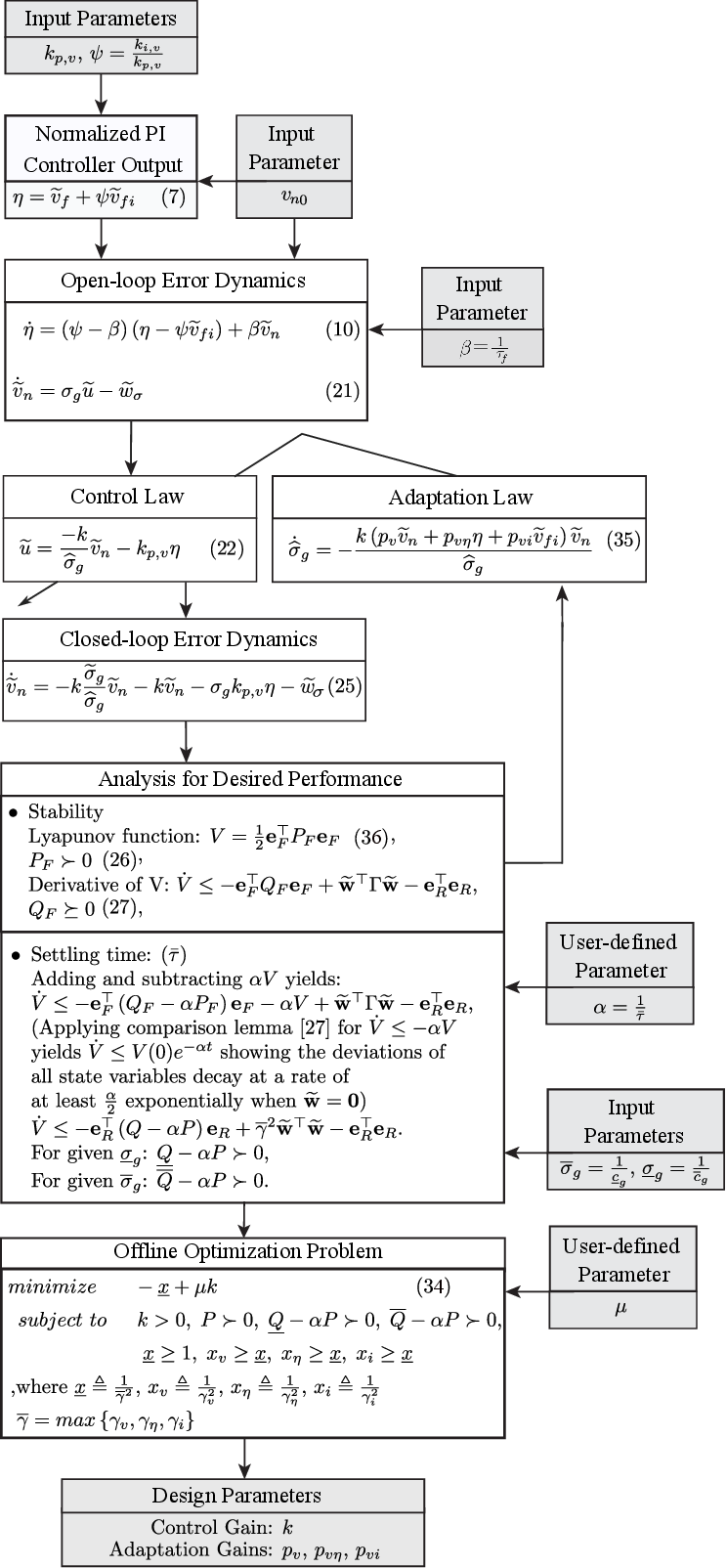}}
    \vspace{-5pt}
    \caption{\sina{Developed robust-adaptive supplementary control design procedure.}}
    \label{fig:control_proced}
\end{figure}
\begin{theorem}\label{Thm: theorem1}
    Assume the optimization problem~\eqref{eqn:optimization prob.} is feasible, then the control law in~\eqref{eqn:Cntrl.Law} along with the adaptive control law in~\eqref{eqn:adaptive law} 
%\begin{align}
%        P>0,\\
%      Q>0,\\
      % \underline{Q}&=\begin{bmatrix}
%\underline{Q}_v&&\underline{Q}_{v\eta}&&Q_{vi}\\ %\star&&\underline{Q}_{\eta}&\underline{Q}_{\eta i}\\
  %\star&&\star&&\underline{Q}_i
      % \end{bmatrix}>0,
      % \\
       %\overline{Q}&=\begin{bmatrix}   %\overline{Q}_v&&\overline{Q}_{v\eta}&&Q_{vi}\\  %\star&&\overline{Q}_{\eta}&&\overline{Q}_{\eta i}\\
        %\star&&\star&&\overline{Q}_i
    %   \end{bmatrix}>0,\\
 %      &Q-\alpha P>0,
  %  \end{align}
  results in $\overline{\gamma}-\text{dissipativity}$ for the closed-loop error dynamics in~\eqref{eqn:closed.loop.er.dyn_final}.
\end{theorem}
\begin{proof}
    See the Appendix. 
\end{proof}

\sina{\begin{remark}
    In principle, it is preferable to consider full matrices for $P_F$, and then $Q_F$ so that there is more flexibility in searching for the control gain $k$; however, in this article $P_F$ is considered of this particular form (as opposed to a full form) in~\eqref{eqn:matrix P} because this greatly simplifies the expressions of elements of $Q$ in~\eqref{eqn:Qv}-\eqref{eqn:Qvi}. As a result, it also reduces the complexity of the optimization problem~\eqref{eqn:optimization prob.} with less decision variables and significantly simplifies the adaptation law in~\eqref{eqn:adaptive law}, which can make its implementation more straightforward.  
\end{remark}}

\nrc{\begin{remark}
    If the optimization problem ~\eqref{eqn:optimization prob.} is feasible for the system under consideration, per Theorem 1, the proposed control law in \eqref{eqn:Cntrl.Law} along with the adaptive control law \eqref{eqn:adaptive law} results in the $\bar{\gamma}$-dissipativity of the closed loop voltage error dynamics described by \eqref{eqn:closed.loop.er.dyn_final}. Notably, this implies that we are not selective in damping a particular mode, but guarantee that the closed-loop voltage error dynamics will not lead to the destabilization of any mode as long as the grid uncertainty model described in Section IV.A. is valid. Therefore, the proposed control will not contribute negatively towards the existing electromechanical modes in the system.
\end{remark}}

The robust adaptive supplementary controller is shown in Fig.~\ref{fig:control_diagram} below in a block diagram form for readers' convenience. \nrc{The control and the adaptation laws involve a few simple operations, which are not computationally expensive.}

\begin{figure}[h!]
    \centerline{\includegraphics[width=0.5\textwidth]{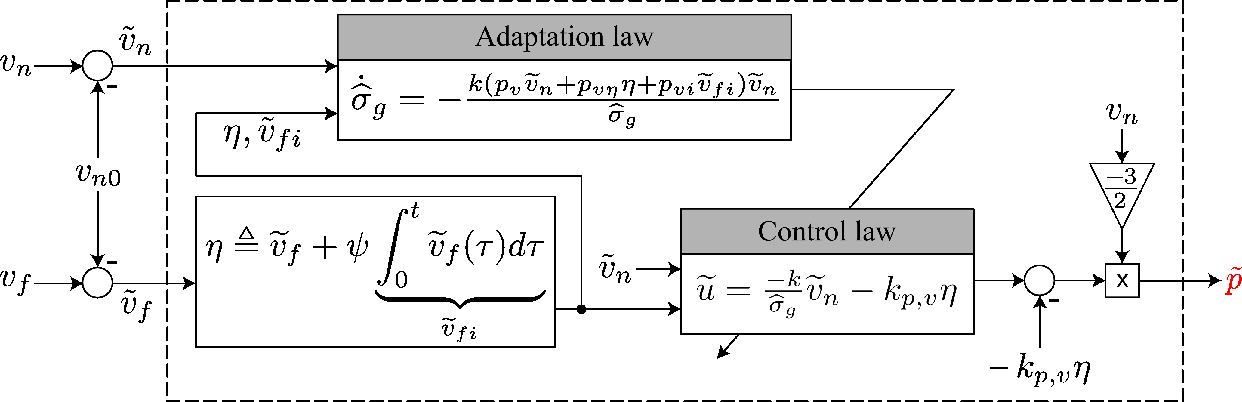}}
    \vspace{-5pt}
    \caption{Block diagram of the proposed controller.}
    \label{fig:control_diagram}
\end{figure}
\section{Case Studies and Results} \label{sec:results}
\subsection{Test Systems}
\nrc{The simulation models employ an averaged model of GFLC along with detailed control loops including the PLL as described in Section~\ref{sec:modeling}, whose parameters are mentioned in the captions of Figs \ref{fig:circuit_gflc}-\ref{fig:gflc_control}.}
As a first step, the effectiveness of the proposed adaptive control law is verified on a space phasor model of a single GFLC connected to an infinite bus through a long transmission line (see Fig.~\ref{fig:UncertaintyMdl}), replicating a weak grid condition. %The simulations employ an averaged model of GFLC along with detailed control loops described in Section~\ref{sec:modeling}. 
The grid has a nominal SCR of about $0.987$, which is very low, and the GFLC generates $700$~MW of power under nominal condition with $r_{g\sina{0}}$ = $0.0140$ pu, $l_{g\sina{0}}$ = $0.1402$ pu, $c_{g\sina{0}}$ = $0.0963$ pu, and $\bar{v}_{gdq}$ = $1$ pu. Linearization around this operating condition reveals the existence of a very poorly-damped SSO mode of $5.36$ Hz as shown in Table~\ref{Tab:test_sys_info}. \nrc{Compass plot of normalized participation factor magnitudes and modeshape angles in Fig.~\ref{fig:mode_shapes_SPC}(a) shows that the dominant states in the SSO mode are the PLL states and the state from the outer voltage control loop delay that oscillate against each other.}

\begin{table}[h!]
\footnotesize
\textcolor{blue}{
\caption{SCR and SSO modes in the test systems}
\label{Tab:test_sys_info}
% \vspace{8pt}
\centering
\begin{tabular}{|L{1.6cm}| C{0.9cm}| C{2.2cm}| *{2}{C{1cm}|}}
    \hline
    \multirow{2}{*}{\textbf{Test system}} & \multirow{2}{*}{\textbf{SCR}} & \multicolumn{3}{c|}{\textbf{SSO mode}} \\
    \cline{3-5}
    & & Mode & \textbf{$f,~Hz$} & \textbf{$\zeta,~\%$} \\
    \hline
    $1$ GFLC & $0.9871$ & $-0.0953\pm33.6562$ & $5.36$ & $0.28$ \\
    \hline
    $2$GFLC-$2$SG & $0.7436^\text{a}$ & $-0.0579\pm39.9471$ & $6.36$ & $0.15$ \\
    \hline
    \lk{$2$GFLC-$16$SG} & \lk{$0.8056^\text{b}$} & \lk{$-0.0468\pm29.7781$} & \lk{$4.74$} & \lk{$0.16$} \\
    \hline
\end{tabular}}
\end{table}
\vspace{-12pt}\noindent \nrc{\footnotesize{$^\text{a}$ SCR seen from bus $6$.}
\noindent \footnotesize{$^\text{b}$ SCR seen from bus $70$.}}
\normalsize

\begin{table}[h!]
\footnotesize
\lk{
\caption{Sensitivity of the SSO mode pertaining to different scenarios}
\label{Tab:sso_sensitivity}
% \vspace{8pt}
\centering
\begin{tabular}{|L{1.6cm}| L{3.8cm}| C{2.2cm}|}
    \hline
    \multirow{1}{*}{\textbf{Test system}} & \multicolumn{1}{c|}{\textbf{Scenario}} & \multicolumn{1}{c|}{\textbf{SSO mode}} \\
    \hline
    $1$ GFLC & Output of GFLC$2$ increased by $0.5$\% & $2.1541\pm30.7419$ \\
    \hline
    $2$GFLC-$2$SG & Outage of one of the parallel lines $9$-$10$ & $-0.2220\pm37.1559$ \\
    \hline
    \lk{$2$GFLC-$16$SG} & \lk{Outage of one of the parallel lines $27$-$53$} & \lk{$0.2623\pm29.3559$} \\
    \hline
\end{tabular}}
\end{table}
\vspace{-5pt}
\normalsize

% \begin{figure}[h!]
%     \centerline{
%     \includegraphics[width=0.24\textwidth]{Figures/SMIB_system.png}}
%     \vspace{-5pt}
%     \caption{Single-GFLC infinite bus system.}
%     \label{fig:SMIB_system}
% \end{figure}

Subsequently, the control law is evaluated in IEEE $2$-area $4$-machine test system \cite{kundur}, replacing G$1$ and G$2$ with identical GFLCs as shown in Fig.~\ref{fig:11_bus_system}. We also introduce two parallel lines between buses $9$ and $10$ while keeping the equivalent parallel impedance same as in the original system. This will be leveraged to simulate line maintenance and line outage scenarios. \textcolor{black}{Both SPC and EMT models are considered.} In each, the SGs are represented by a $8$th-order model including stator transients. Moreover, dynamic models of constant-impedance loads are considered. \textcolor{black}{The transmission lines are dynamically modeled using lumped $\pi$-section and Bergeron models for SPC and EMT, respectively}. Averaged models of GFLCs are used in EMT as in SPC. Linearization of SPC model around nominal GFLC and SG power outputs of $700$ MW each and tie-flow condition of $400$ MW shows an extremely poorly-damped $6.36$ Hz mode with nominal SCR of $0.744$ corresponding to GFLC$2$ as measured at bus $6$, see Table~\ref{Tab:test_sys_info}. \nrc{As before, the compass plot in Fig.~\ref{fig:mode_shapes_SPC}(b) shows that PLL states of both GFLCs oscillate in synch against those of the outer voltage delay states. This indicates that the SSO mode is indeed a weak-grid mode, not an inter-IBR mode. The proposed supplementary control is exclusively applied in GFLC$2$ as the dominant states contributing to the corresponding SSO mode originate from GFLC$2$, see Fig.~\ref{fig:mode_shapes_SPC}(b).} \nrc{This is further supported by the residue magnitudes of the SSO mode that are \lk{$0.0117$} and \lk{$0.0120$} for GFLC1 and GFLC2, respectively. }

\begin{figure}[h!]
    \centerline{
    \includegraphics[width=0.49\textwidth]{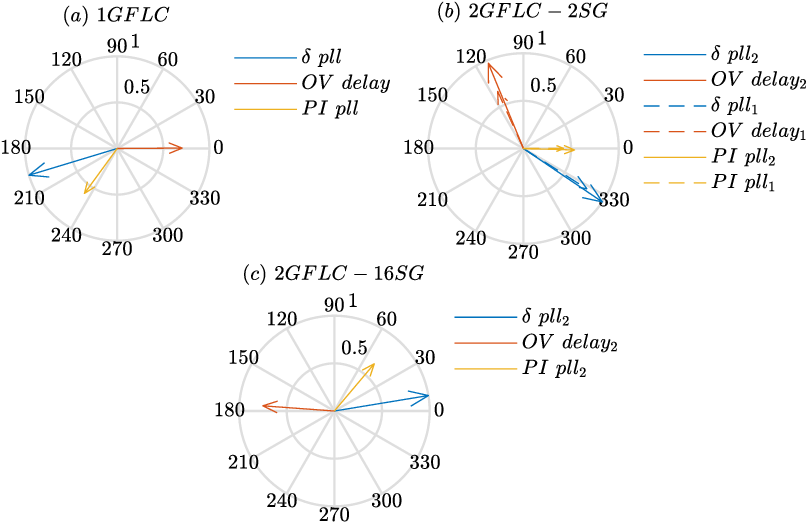}}
    \vspace{-5pt}
    \caption{Compass plots of normalized participation factor magnitudes and modeshape angles of the dominant states contributing to the poorly-damped SSO mode in (a) single-GFLC-infinite-bus system (Fig.~\ref{fig:UncertaintyMdl}), (b) $2$-area $4$-machine test system with $2$ GFLCs (Fig.~\ref{fig:11_bus_system}), (c) modified $5$-area $68$-bus test system with $2$ GFLCs (Fig.~\ref{fig:68_bus_system}). OV: outer voltage control.}
    \label{fig:mode_shapes_SPC}
    \vspace{-5pt}
\end{figure}

\begin{figure}[h!]
    \centerline{
    \includegraphics[width=0.44\textwidth]{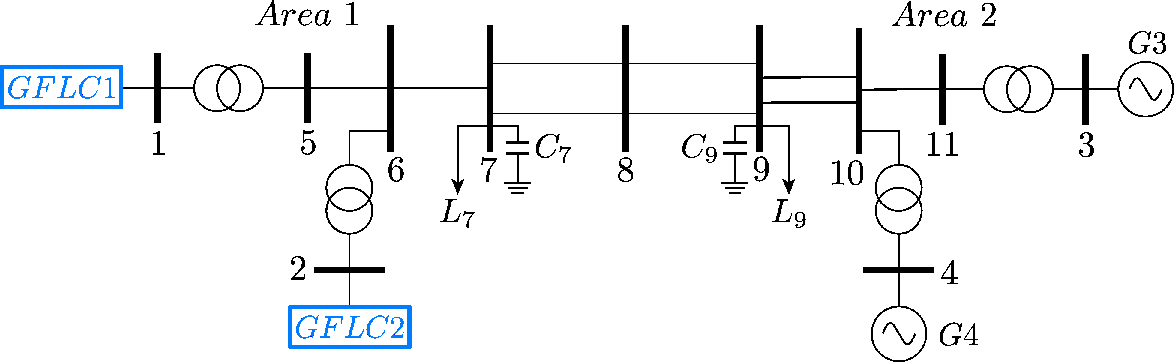}}
    \vspace{-5pt}
    \caption{Modified $2$-area test system with 50\% IBR penetration.}
    \label{fig:11_bus_system}
    \vspace{-5pt}
\end{figure}

% Both test systems exhibit a poorly damped SSO mode for the system parameters specified in Section~\ref{sec:modeling}.

\begin{figure}[h!]
    \centerline{
    \includegraphics[width=0.45\textwidth]
    {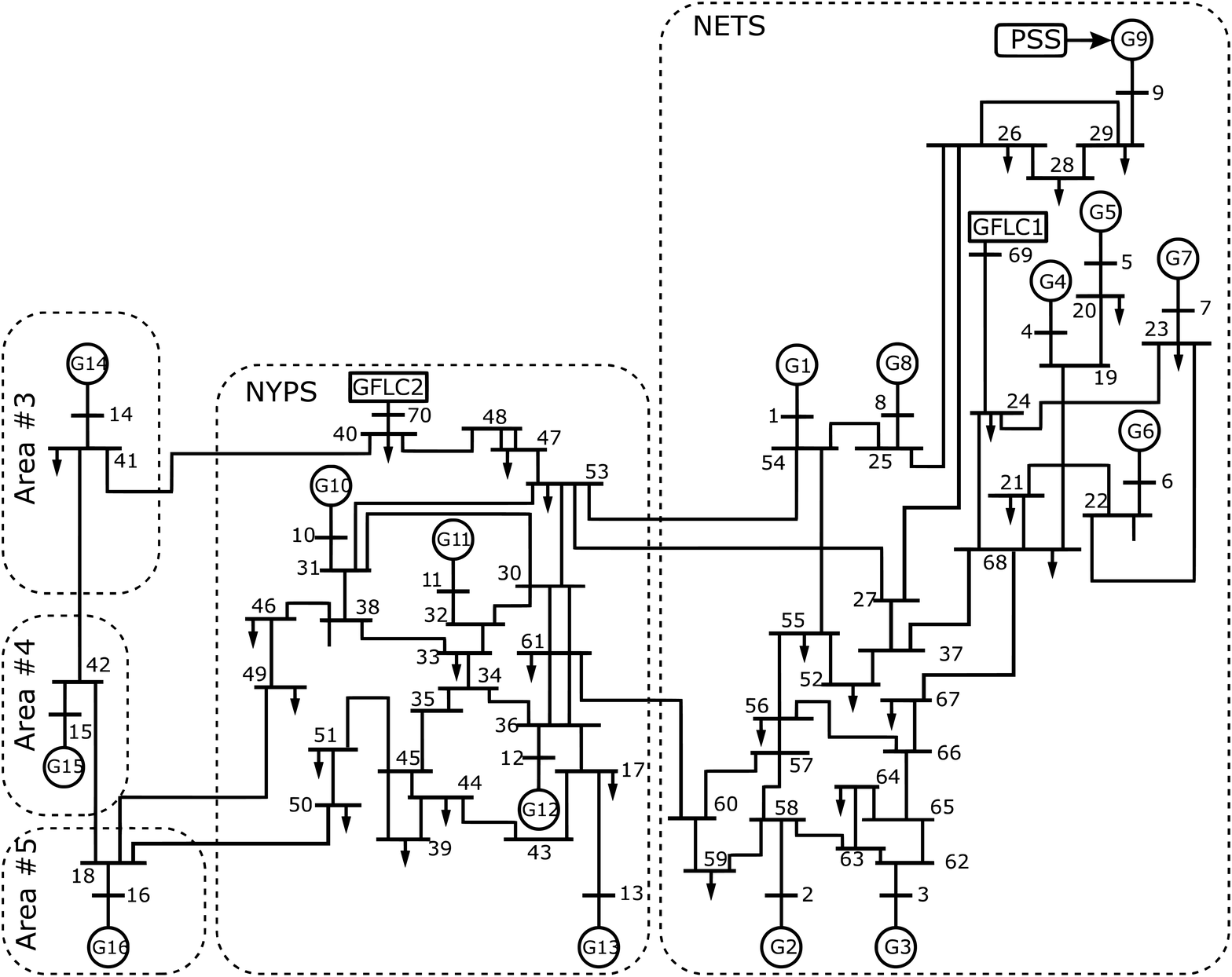}}
    \vspace{-5pt}
    \caption{\nrc{Modified $5$-area $68$-bus test system with $2$ IBRs connected to buses $24$ (GFLC1) and $40$ (GFLC2) using dedicated transmission lines.}}
    \label{fig:68_bus_system}
\end{figure}

\nrc{Finally, we consider the IEEE $5$-area $68$-bus New York-New England system \cite{nilthesis}. The test system was modified by connecting two GFLCs at buses $24$ and $40$ using two transmission lines $24$-$69$ and $40$-$70$, respectively. The QPC model of the system was considered with $6$th-order subtransient model of SGs and algebraic model of the transmission lines and loads. Generators $G1$-$G8$ are equipped with dc exciters, $G9$ has a static exciter with a power system stabilizer (PSS), whereas $G10$-$G16$ have manual excitation. Averaged GFLC models are used as in the previous systems. Modal analysis reveals a poorly-damped SSO mode of frequency $4.74$ and damping ratio $0.0016$ with significant participation from GFLC$2$ connected to bus $70$, see Fig.~\ref{fig:68_bus_system}. The SCR from bus $70$ is $0.8056$, which is shown in Table II.} \nrc{Residue magnitudes of the SSO mode are \lk{$0.0001$} and \lk{$0.0371$} for GFLC1 and GFLC2, respectively, which led to the choice of GFLC2 as the actuator for supplementary control. Finally, Table III reveals the sensitivity of the SSO modes with respect to changes in operating conditions. These changes include increase in GFLC power output and outage of certain lines.}

The test systems are implemented and simulated in Matlab/Simulink \cite{simulink} environment. \textcolor{black}{In addition, the $2$-area test system is also implemented in EMTDC/PSCAD \cite{pscad}.} We investigate multiple cases to assess the viability of the control law, verifying its performance under realistic disturbance/fault conditions within test systems. For each case, we present three scenarios: (1) no supplementary control of GFLC (denoted by $NC$), (2) classical linear state feedback-based supplementary control of GFLC discussed in Section~\ref{sec:Numerical} (denoted by $WC - State\ feedback$), and (3) robust adaptive supplementary control of GFLC proposed in Section~\ref{sec:ControlDevelop} (denoted by $WC - Adaptive$). Due to space restrictions, only a selection of the case studies is discussed in detail.

\subsection{Numerical Design Parameters}\label{sec:Numerical}
The adaptive control design parameters are chosen so that the constraints in the optimization problem~\eqref{eqn:optimization prob.} are satisfied. For both single-GFLC and multi-GFLC system, the maximum allowable settling time is $\bar{\tau
}=15$ s, which gives $\alpha=0.066$. For the single-GFLC system, the uncertainty bounds on $\sina{\sigma_g}$ are given by $8.1487\leq\sina{\sigma_g}\leq 10.3896$ pu. Given $\mu=2$, and then solving the optimization problem in~\eqref{eqn:optimization prob.} yields: $p_v=0.0534$, $p_{v\eta}=3.7221$, $p_{vi}=-0.0034$, $k=1440.8$, and $\overline{\gamma}=1$. For \nrc{both} multi-GFLC-multi-machine systems, we consider $25.38\leq\sina{\sigma_g}\leq 32.68$ pu. In this case, for given $\mu=0.5$, the design parameters are obtained as follows: $p_v=0.0034$, $p_{v\eta}=0.0747$, $p_{vi}=0.00019$, $k=16002$, and $\overline{\gamma}=0.4163$. 

\noindent \nrc{\textit{Classical state-feedback controller:} The design of the state-feedback controller is performed first by linearizing the nonlinear model provided by~\eqref{eqn:nerr_dyn}, and then feeding back \sina{${v}_n$} to the controller. The gain of the controller are chosen $k_{sf}=-22$ for the single-GFLC system and $k_{sf}=-350$ for the  multi-GFLC-multi-machine system so that the state-feedback controller generates the control signal modulation $\widetilde{p}$. The block diagram of the state-feedback supplementary controller is shown in Fig.~\ref{fig:State-feedback controller}. The term $\frac{3 c_{g0}v_{n0}}{2}$ is due to linearization, where $c_{g0}$ is the nominal value of the grid capacitance and $v_{n0}$ is the pre-disturbance nominal voltage.}

% \begin{figure}[h!]
%     \centerline{\includegraphics[width=0.27 \textwidth]{Figures/StateFeedback.eps}}
%     \vspace{-5pt}
%     \caption{\sina{State-feedback supplementary controller.}}
%     \label{fig:State-feedback controller}
% \end{figure}

\begin{figure}[h!]
    \centerline{\includegraphics[width=0.51 \textwidth]{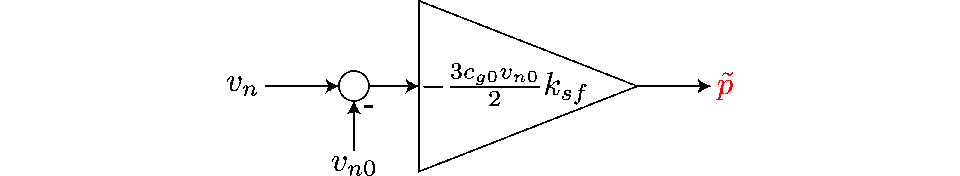}}
    \vspace{-5pt}
    \caption{\sina{State-feedback supplementary controller.}}
    \label{fig:State-feedback controller}
\end{figure}

\subsection{Case Studies - SPC Simulations}
\subsubsection{Single-GFLC Test System}
% The test system comprises of a single GFLC connected to an infinite bus through a long transmission line segment (see Fig.~\ref{fig:SMIB_system}), replicating a weak grid condition. The grid has a short circuit ratio of about $0.$, and the GFLC generates $700$~MW of power under nominal conditions. Here, 
We examine how much increase in power output from the GFLC can be achieved in this weak grid under uncertainty.%s two different disturbances applied to the test system.

% \noindent \textit{a) A $0.5\%$ step increase in real power:} Figure~\ref{fig:SM_step} presents a comparison of dynamic responses under previously-mentioned three scenario for a $0.5$\% step increase in GFLC active power reference. The results indicate that NC scenario leads to instability, whereas both WC scenarios successfully maintain system stability. Notably, the adaptive control demonstrates more efficient damping compared to state-feedback control upon a change in the system operating point.
% \begin{figure}[h!]
%     \centerline{
%     \includegraphics[width=0.45\textwidth]{Figures/SingleMachine/SM_step.eps}}
%     \vspace{-5pt}
%     \caption{Single GFLC system: GFLC responses following a $0.5$\% step change in active power reference for scenarios (i) without supplementary control, (ii) with state-feedback control, and (iii) with proposed adaptive control. }
%     \label{fig:SM_step}
% \end{figure}

\noindent \textit{A $0.5\%$ step increase in real power under uncertainty in $c_g$:}  Figure~\ref{fig:SM_uncertainty_step} illustrates the system response following a $0.5$\% step increase in GFLC active power reference while ramping up $c_g$ from $0.0963$ pu to $0.2454$ pu ($27.5\%$ change) during $t = 1$ s - $11$ s emulating uncertainty in the grid. The results  highlight that the adaptive control significantly enhances the damping of SSOs under this condition. Moreover, the maximum value of the control input $\widetilde{p}$ is only about $0.15\%$ of the nominal power, which is negligible. The slow convergence of the estimated parameter $\sina{\widehat{\sigma}_g}$ is also shown in Fig.~\ref{fig:SM_uncertainty_step}, which is governed by the adaptation law~\eqref{eqn:adaptive law}.
\begin{figure}[h!]
    \centerline{
    \includegraphics[width=0.42\textwidth]{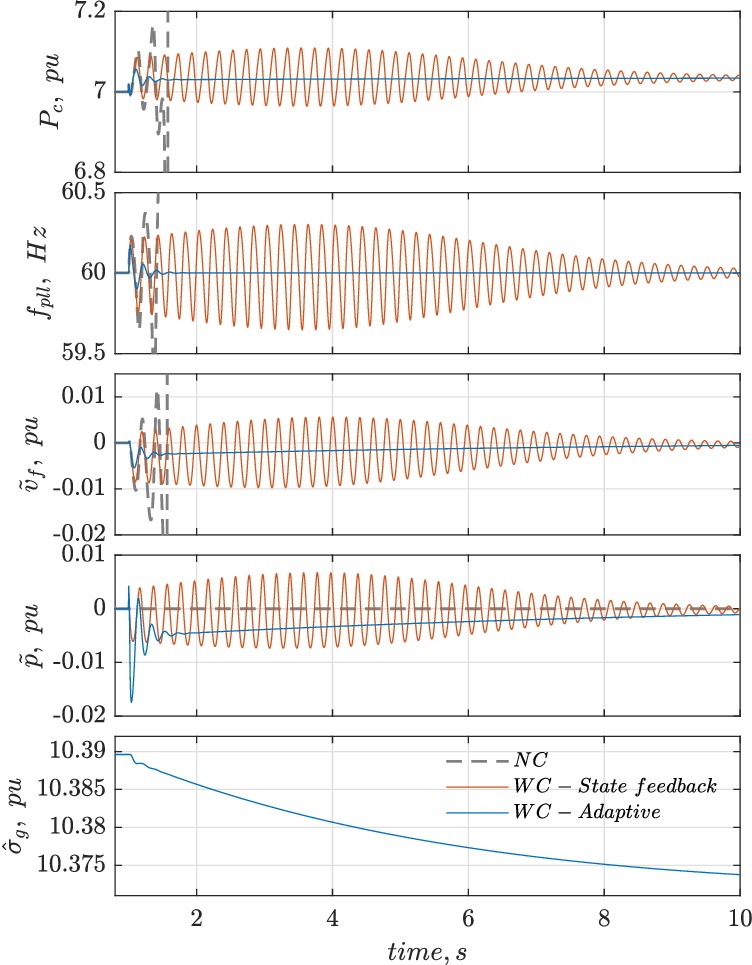}}
    \vspace{-5pt}
    \caption{Single-GFLC system: GFLC responses following a $0.5$\% step increase in active power reference. 
    %for scenarios (i) without supplementary control, (ii) with state-feedback control, and (iii) with proposed adaptive control. 
    Uncertainty in power grid is emulated by ramping up $c_g$ from $0.0963$ pu to $0.2454$ pu ($27.5\%$ change) during $t = 1$ s - $11$ s.}
    \label{fig:SM_uncertainty_step}
    \vspace{-8pt}
\end{figure}

\subsubsection{Multi-GFLC-\nrc{Modified $2$-Area} Test System}
% It is noted that the mutil-machine system (see Fig.~\ref{fig:11_bus_system}) operates under nominal condition with $400$~MW tie line flow and each GFLC produces $700$~MW of power. 
Three cases are explored to validate the performance of the proposed controller that focus on two types of studies: (1) examine how much increase in power output from GFLC$2$ can be accommodated in the nominal system and also when a line is out under maintenance, thereby reducing the system strength further; and (2) consider large-signal disturbances like short-circuit faults.

\noindent \textit{a) Step increase in real power from GFLC$2$:} Dynamic responses of GFLC$2$ to a $1$\% step increase followed by an additional $4$\% step increase in GFLC$2$ active power reference, respectively at $t = 1$ s and $t = 5$~s, are illustrated in Fig.~\ref{fig:MM_step}. The figure clearly demonstrates that the NC scenario is unstable after the initial $1$\% step increase. %and both WC scenarios are stable. %In contrast, following the subsequent $4$\% step increase, adaptive control successfully stabilizes the system, whereas the state feedback control fails to do so. 
It also shows that the proposed control allows a $5$\% increase in real power output from GFLC$2$, whereas state-feedback control allows only $1$\% increase.

\begin{figure}[h!]
    \centerline{
    \includegraphics[width=0.405\textwidth]{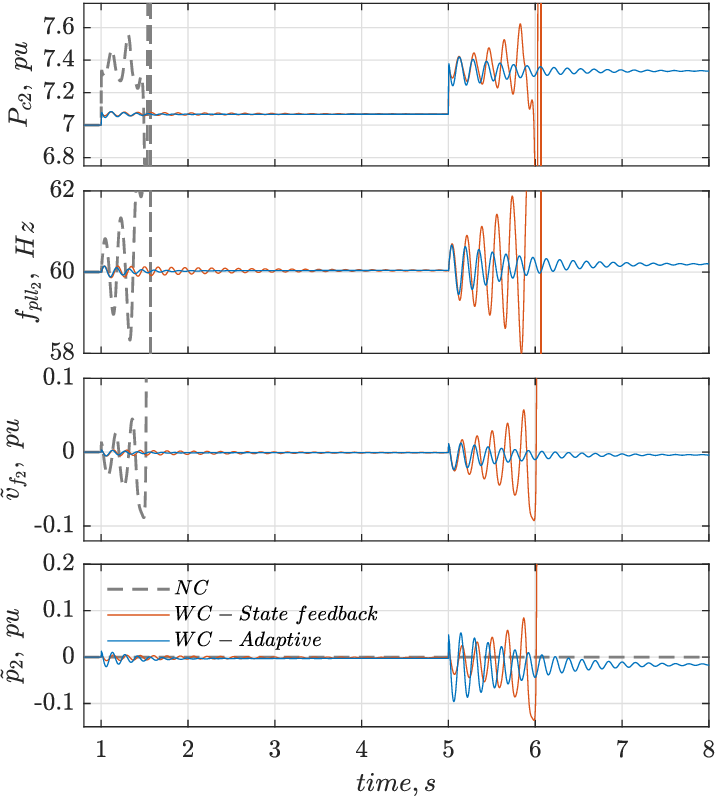}}
    \vspace{-5pt}
    \caption{Multi-GFLC \nrc{modified $2$-area} system: GFLC$2$ responses after a $1$\% step increase followed by an additional $4$\% step increase in active power reference. 
    % respectively at $t = 1$ s and $t = 5$~s. for scenarios (i) without supplementary control, (ii) with state-feedback control, and (iii) with proposed adaptive control.
    }
    \label{fig:MM_step}
    \vspace{-12pt}
\end{figure}

\noindent \textit{b) Step increase in real power from GFLC$2$ when system under line maintenance:} One of the often-overlooked challenges come from the system operating under line maintenance condition leading to a weaker grid. To emulate that we consider absence of one of the parallel lines $9-10$ in Fig.~\ref{fig:11_bus_system} under \textit{nominal} conditions. Two consecutive step changes in $p_{ref}$ of GFLC$2$ under this condition shows that the adaptive control can increase real power by $2$\%, whereas state-feedback control can achieve a $1\%$ increase, see Fig.~\ref{fig:MM_linemain_plus_step}.

\begin{figure}[h!]
    \centerline{
    \includegraphics[width=0.405\textwidth]{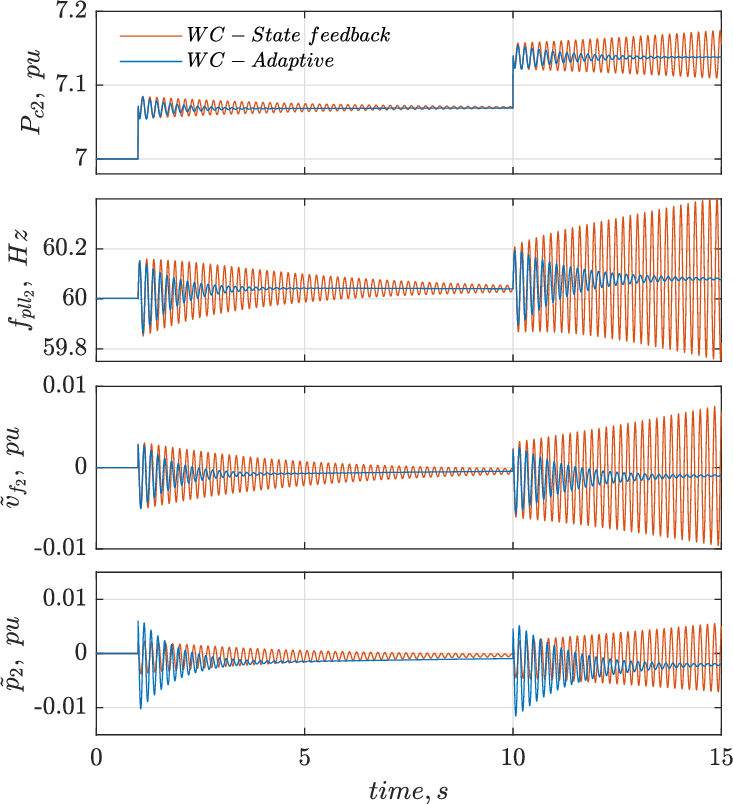}}
    \vspace{-5pt}
    \caption{Multi-GFLC \nrc{modified $2$-area} system under line maintenance: GFLC$2$ response following two consecutive $1$\% step increases in active power reference at $t = 1$ s and $t = 10$ s. One of the parallel lines between $9$-$10$ in Fig.~\ref{fig:11_bus_system} is under maintenance.}
    \label{fig:MM_linemain_plus_step}
    \vspace{-12pt}
\end{figure}

\noindent \textit{c) Short circuit fault followed by line outage:} A $3$-phase short circuit fault is applied near bus $10$ using a $15$ pu fault admittance for one cycle followed by the outage of one of the parallel lines $9$-$10$. Figure~\ref{fig:MM_fault_plus_lineout} shows the dynamic behavior of GFLC$2$ along with the center of inertia frequency ($f_{COI}$) of the system. It can be seen that the system is stable and well-damped in presence of adaptive supplementary control, whereas it becomes unstable in two other scenarios. \vspace{-3pt}

\begin{figure}[h!]
    \centerline{
    \includegraphics[width=0.42\textwidth]{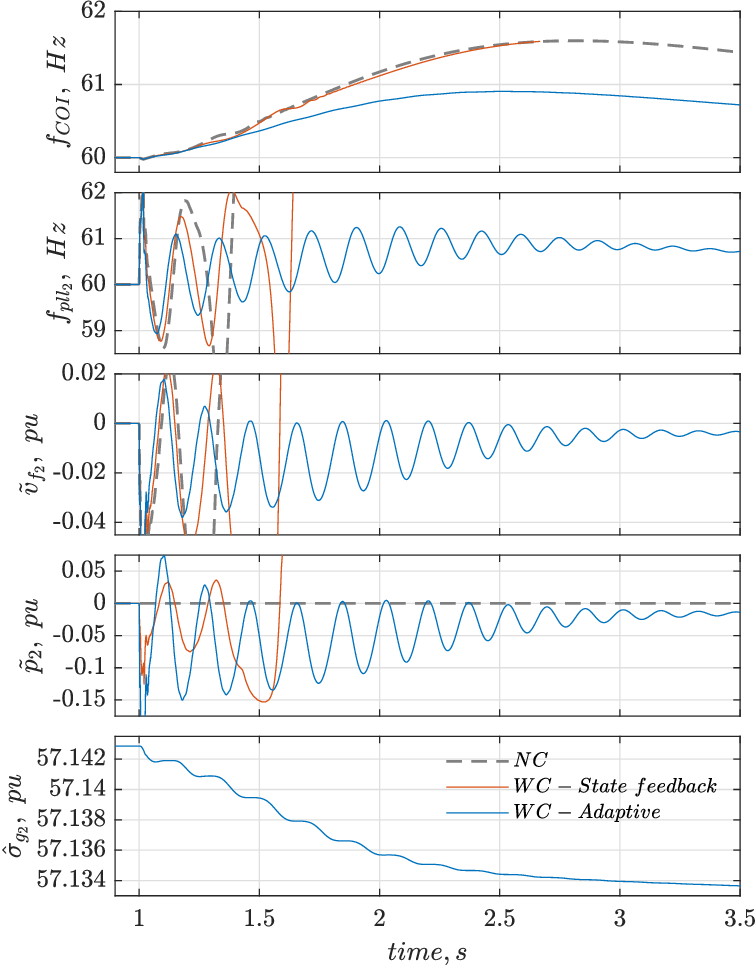}}
    \vspace{-5pt}
    \caption{Multi-GFLC \nrc{modified $2$-area}  system: System response after a three-phase one-cycle short circuit fault near bus $10$ using a $15$ pu fault admittance followed by the outage of one of the parallel lines $9$-$10$ in Fig.~\ref{fig:11_bus_system}.}
\label{fig:MM_fault_plus_lineout}
\vspace{-15pt}
\end{figure}

\begin{figure}[h!]
    \centerline{
    \includegraphics[width=0.42\textwidth]{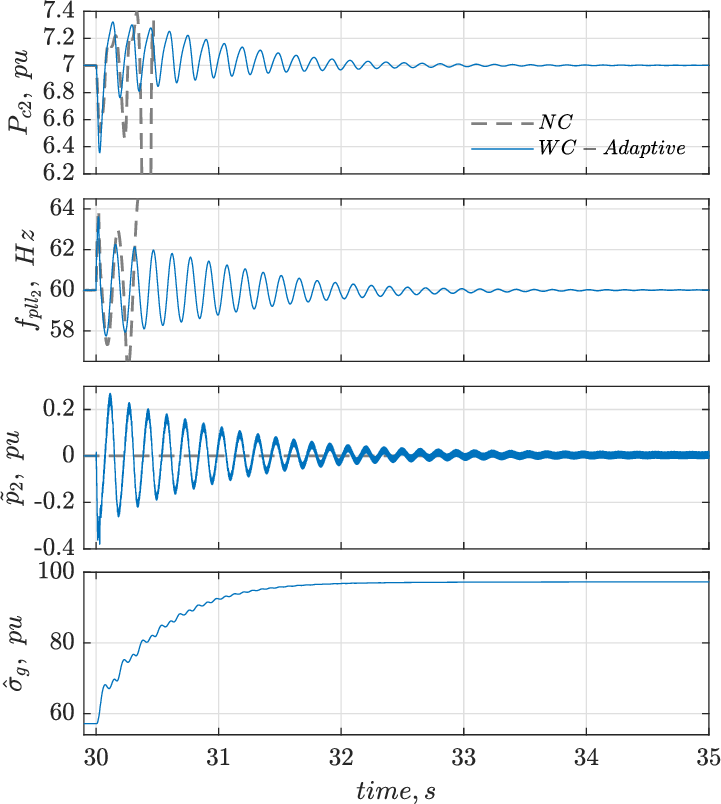}}
    \vspace{-5pt}
    \caption{EMT model of Multi-GFLC \nrc{modified $2$-area} system: Response following a three-phase one-cycle short circuit fault near bus $10$ using a $33$ pu fault admittance.}
    \label{fig:MM_fault_EMT}
    \vspace{-10pt}
\end{figure}

\begin{figure}[h!]
    \centerline{
    \includegraphics[width=0.42\textwidth]{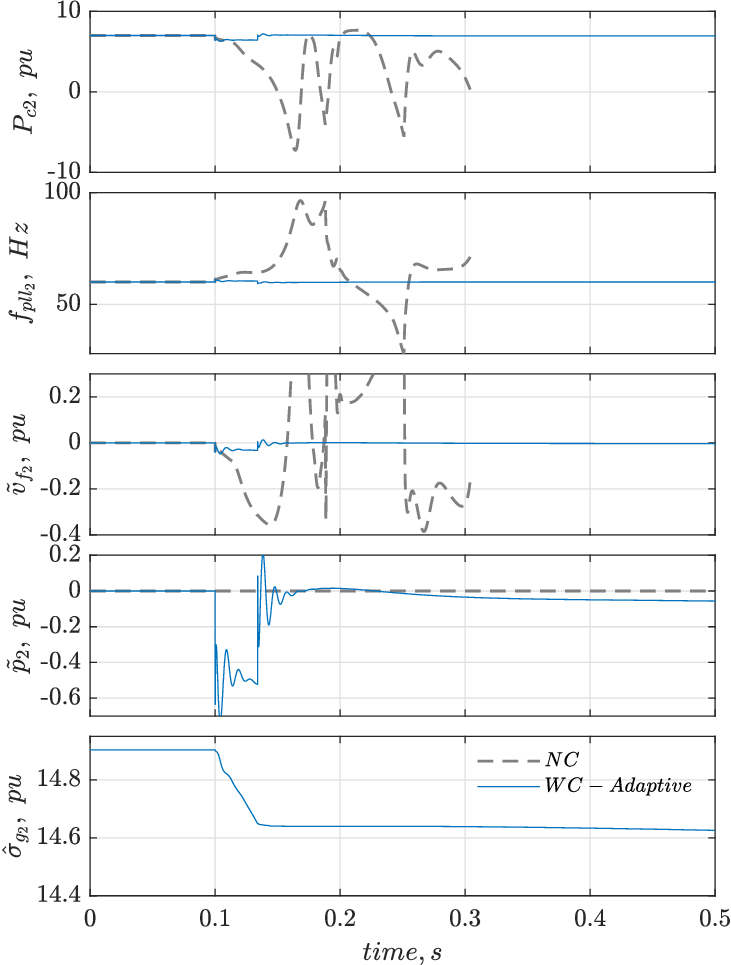}}
    \vspace{-5pt}
    \caption{\nrc{Multi-GFLC modified $68$-bus system: Zoomed view of GFLC2 response following a self-clearing three-phase two-cycle  fault near bus $50$ (Fig. \ref{fig:68_bus_system}) using a $200$ pu fault admittance.}}
    \label{fig:SCF_bus50Zoom}
    \vspace{-10pt}
\end{figure}

\begin{figure}[h!]
    \centerline{
    \includegraphics[width=0.42\textwidth]{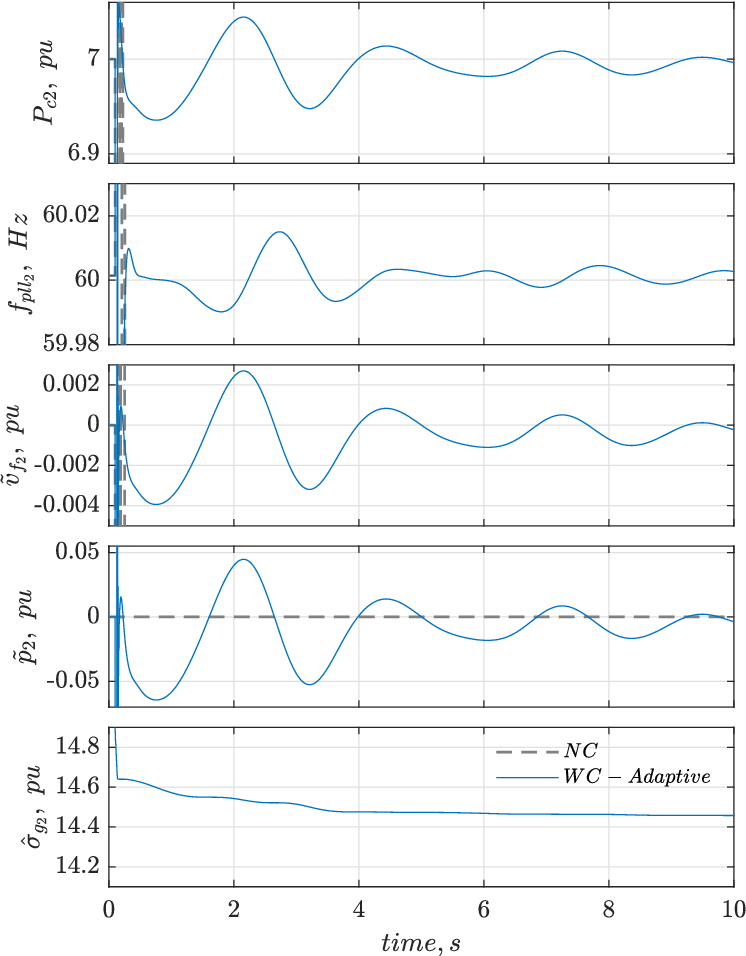}}
    \vspace{-5pt}
    \caption{\nrc{Multi-GFLC modified $68$-bus system: Response following a self-clearing three-phase two-cycle  fault near bus $50$ (Fig. \ref{fig:68_bus_system}) using a $200$ pu fault admittance.}}
    \label{fig:SCF_bus50}
    \vspace{-10pt}
\end{figure}

\begin{figure}[h!]
    \centerline{
    \includegraphics[width=0.42\textwidth]{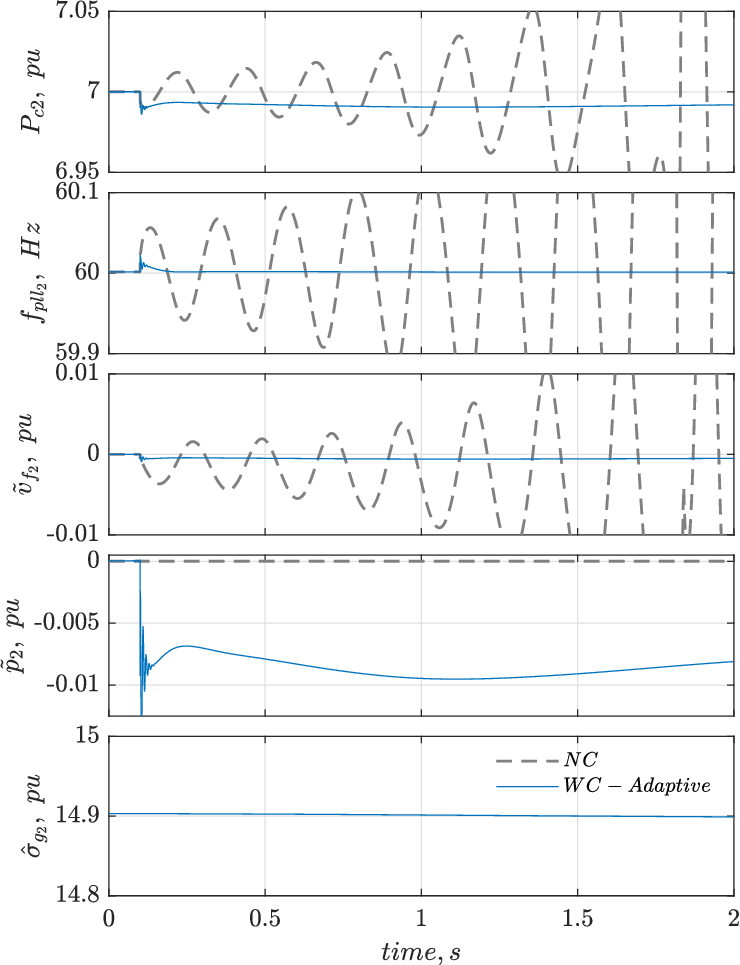}}
    \vspace{-5pt}
    \caption{\nrc{Multi-GFLC modified $68$-bus system: Response following the outage of one of the parallel circuits in $27$-$53$ tie-line of Fig.~\ref{fig:68_bus_system}.}}
    \label{fig:Lineout_27_53_zoomed}
    \vspace{-10pt}
\end{figure}

\subsection{Case Studies - EMT Simulations}
\textcolor{black}{\textit{Self-clearing short circuit fault:} A $3$-phase self-clearing fault lasting one cycle ($17$ ms) with a fault admittance of $33$ pu is applied at bus $10$ of the multi-GFLC \nrc{modified $2$-area} test system implemented in EMT domain. The simulations employ a $5$ ms moving average filtering on $v_{n}$ for the supplementary control action. %effectively smoothing out high-frequency transients. 
The EMT responses following the fault are illustrated in Fig.~\ref{fig:MM_fault_EMT}. The results indicate that NC scenario is unstable, whereas the system response with the adaptive control is stable and well-damped. Although the controller could not stabilize the system following more severe faults, this outcome, nonetheless, demonstrates the efficacy of the proposed supplementary control in stabilizing against large-signal disturbances. Notably, unlike the single-GFLC system, the estimated parameter $\sina{\widehat{\sigma}_{g}}$ in multi-GFLC system is initialized with an estimate, assuming it is unknown. This leads to the parameter settling at a different value even after the system arrives back to the same pre-disturbance equilibrium. Moreover, a contrasting trend can be observed for $\sina{\widehat{\sigma}_{g}}$ in Figs~\ref{fig:MM_fault_plus_lineout} and \ref{fig:MM_fault_EMT} during the three-phase fault condition, which is believed to be caused by different representations of transmission network adopted in SPC (lumped $\pi$-section) and EMT (Bergeron model) simulations.}

\subsection{\nrc{Case Studies - QPC Simulations}}\label{sec:68busCases}
\nrc{Two types of studies are performed in the modified $68$-bus system shown in Fig.~\ref{fig:68_bus_system}: (1) large-signal disturbance using a self-clearing short circuit fault and (2) a tie-line outage with the objective of making the system weaker than the nominal condition.}

\noindent \nrc{\textit{a) Self-clearing short circuit fault:} First, we consider a self-clearing three-phase $2$-cycle short circuit fault near bus $50$ of Fig.~\ref{fig:68_bus_system} using a $200$ pu fault admittance at $t$ = $0.1$ s. Figure~\ref{fig:SCF_bus50Zoom} shows the dynamic behavior of GFLC2 for $0.4$ s following the fault. It can be seen that the fault is severe enough to make GFLC2 unstable in absence of supplementary damping control. On the other hand, the proposed supplementary adaptive controller stabilizes GFLC2. It should be noted that although the maximum variation in $\widetilde{p}_2$ is about $10\%$ of the nominal power immediately after the large disturbance, it becomes much smaller within a very short period of time. Figure~\ref{fig:SCF_bus50} shows the extended term simulation result emphasizing that the SSO mode have been damped and only the electromechanical modes are observed in the response.}

\noindent \nrc{\textit{b) Outage of one of the parallel circuits in tie-line $27$-$53$:} Next, we consider outage of one of the parallel circuits in tie-line $27$-$53$ in Fig.~\ref{fig:68_bus_system}. It can be seen from Fig.~\ref{fig:Lineout_27_53_zoomed} that GFLC2 becomes unstable in absence of the supplementary control. The negatively damped unstable SSO mode is visible in the response of GFLC2. In contrast, the adaptive supplementary controller is able to stabilize this unstable oscillation very quickly with very small amount of modulating signal $\widetilde{p}_2$.  } \\

\begin{figure}[h!]
    \centerline{
    \includegraphics[width=0.42\textwidth]{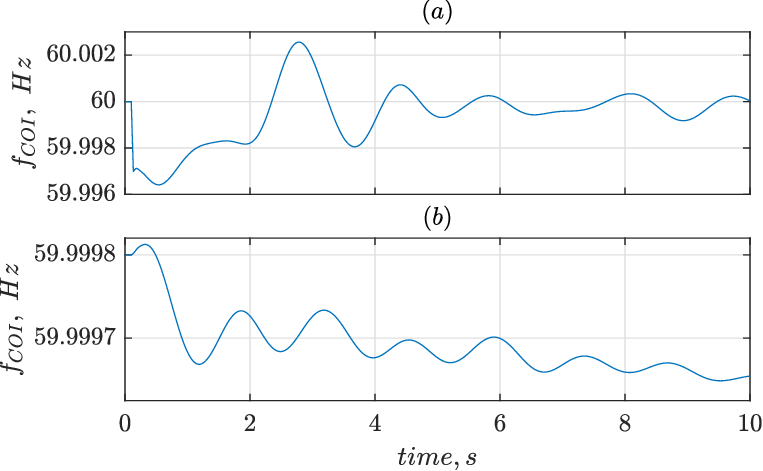}}
    \vspace{-5pt}
    \caption{\lk{Multi-GFLC modified $68$-bus system: Response following (a) a self-clearing three-phase two-cycle  fault near bus $50$ using a $200$ pu fault admittance and (b) the outage of one of the parallel circuits in $27$-$53$ tie-line.}}
    \label{fig:combined_plot_68bus}
    \vspace{-5pt}
\end{figure}

\noindent \nrc{\textit{Observation regarding existing electromechanical modes:} The extended term simulation in Fig.~\ref{fig:SCF_bus50} shows the signature of low frequency electromechanical modes in the response of the system. As expected, damping of these modes are not negatively impacted. To emphasize the point, we have plotted the center of inertia (COI) frequency of the SGs in presence of the adaptive controller for both of the case studies that can be seen from Fig.~\ref{fig:combined_plot_68bus}.}

\section{Conclusions and Future Work}
In this paper, a robust adaptive supplementary control with stability certificate was proposed for damping SSOs involving GFLCs, especially under weak grid scenarios. The control law was derived for one GFLC connected to a grid represented by an ideal voltage source behind an impedance with uncertainties and disturbances in the system. The analysis provided sufficient conditions for stability and showed that the closed-loop system is $dissipative$. It provided the minimum of the upper bound on the $\mathcal{L}_2$-gain from the lumped disturbances to the output system errors while preventing the gain of the supplementary control from becoming very high via solving an optimization problem. The effectiveness of the proposed controller was validated using several case studies performed on the SPC model of a single-GFLC-infinite-bus system, both SPC and EMT models of the IEEE $2$-area test system with two GFLCs replacing two SGs, \nrc{and a QPC model of the modified IEEE $5$-area test system with two GFLCs}. Simulation results show that the proposed robust adaptive supplementary controller can perform damping action effectively in stabilizing SSOs while demonstrating a significant degree of robustness towards the disturbances and uncertainties in the grid.

Future work will focus on extending this analysis to decentralized control for multiple GFLCs and incorporating current limits in the inner loops during control design. Additionally, we plan to estimate disturbances to feed them back into the control law, allowing for reduced supplementary control gain at the expense of increasing controller complexity. %\sina{Furthermore, inter-area SSO modes problem is another challenging phenomena that we aim to solve.}

%In this paper the stability analysis to develop a robust adaptive supplementary control is performed only for one GFLC. The next work is providing the stability analysis of a decentralized control structure for large-scale GFLCs to damp SSOs while each of which is using the developed robust adaptive supplementary control in this article. Moreover, in this paper, we did not consider the current limit in the inner loop, which in the next work, we will consider it in the analysis. Furthermore, we approached the problem of disturbance using the concept of $dissipativity$ via minimizing the upper bound on the $\mathcal{L}_2$-gain of the closed-loop system which made the gain of the controller increased so that the closed-loop system becomes robust. In the next work, we estimate the disturbance and then feeding it back to the control law. By doing this, we can reduce the gain of the supplementary control yet the controller becomes more complicated.   
\bibliographystyle{IEEEtran}
\bibliography{R1}
\section*{Appendix}
\textbf{Proof of Theorem 1:}
    Consider the following candidate positive definite energy function
\begin{align}\label{eqn:Lyapunov func.}
        V=\frac{1}{2}\begin{bmatrix}
            \widetilde{v}_n&
            \eta&
            \widetilde{v}_{fi}&
            \widetilde{\sina{\sigma}}_g
        \end{bmatrix}
       P_F\begin{bmatrix}
            \widetilde{v}_n&
            \eta &
            \widetilde{v}_{fi}&
            \widetilde{\sina{\sigma}}_g
        \end{bmatrix}^\top.
    \end{align}
    Taking the first-time derivative of~\eqref{eqn:Lyapunov func.} yields
    \begin{align*}
        \dot{V}=&p_v\widetilde{v}_n\dot{\widetilde{v}}_n+p_{\eta}\eta\dot{\eta}+p_i\widetilde{v}_{fi}\widetilde{v}_f-\widetilde{\sina{\sigma}}_g\dot{\widehat{\sina{\sigma}}}_g\\&+p_{v\eta}\widetilde{v}_n\dot{\eta}+p_{v\eta}\eta\dot{\widetilde{v}}_n+p_{vi}\widetilde{v}_n\widetilde{v}_f+p_{vi}\widetilde{v}_{fi}\dot{\widetilde{v}}_n\\
        &+p_{\eta i}\eta\widetilde{v}_f+p_{\eta i}\widetilde{v}_{fi}\dot{\eta}.
    \end{align*}
Inserting the filtered tracking in~\eqref{eqn:er+inter}, and its dynamics in~\eqref{eqn:filt.track.dyn} and the closed-loop error dynamics in~\eqref{eqn:closed.loop.er.dyn_final}
    \begin{align*}
        \dot{V}=&p_v\widetilde{v}_n\left(-k\frac{\widetilde{\sina{\sigma}}_g}{\widehat{\sina{\sigma}}_g}\widetilde{v}_n-k\widetilde{v}_n-\sina{\sigma_g}k_{p,v}\eta-\widetilde{w}_{\sina{\sigma}}\right)\\ &+p_{\eta}\eta\left(\left(\psi-\sina{\beta}\right)\left(\eta-\psi\widetilde{v}_{fi}\right)+\sina{\beta}\widetilde{v}_n\right)+p_i\widetilde{v}_{fi}\left(\eta-\psi\widetilde{v}_{fi}\right)\\
        &-\widetilde{\sina{\sigma}}_g\dot{\widehat{\sina{\sigma}}}_g+p_{v\eta}\widetilde{v}_n\left(\left(\psi-\sina{\beta}\right)\left(\eta-\psi\widetilde{v}_{fi}\right)+\sina{\beta}\widetilde{v}_n\right)\\
        &+p_{v\eta}\eta\left(-k\frac{\widetilde{\sina{\sigma}}_g}{\widehat{\sina{\sigma}}_g}\widetilde{v}_n-k\widetilde{v}_n-\sina{\sigma_g}k_{p,v}\eta-\widetilde{w}_{\sina{\sigma}}\right)\\
&+p_{vi}\widetilde{v}_n\left(\eta-\psi\widetilde{v}_{fi}\right)\\
        &+p_{vi}\widetilde{v}_{fi}\left(-k\frac{\widetilde{\sina{\sigma}}_g}{\widehat{\sina{\sigma}}_g}\widetilde{v}_n-k\widetilde{v}_n-\sina{\sigma_g}k_{p,v}\eta-\widetilde{w}_{\sina{\sigma}}\right)\\
        &+p_{\eta i}\eta\left(\eta-\psi\widetilde{v}_{fi}\right)+p_{\eta i}\widetilde{v}_{fi}\left(\left(\psi-\sina{\beta}\right)\left(\eta-\psi\widetilde{v}_{fi}\right)+\sina{\beta}\widetilde{v}_n\right).
       % =&-p_vk\frac{\widetilde{\theta}_g}{\widehat{\theta}_g}\widetilde{v}_n^2-p_vk\widetilde{v}_n^2-p_vk_{p,v}\sina{\sigma_g}\widetilde{v}_n\eta-p_v\widetilde{v}_n\widetilde{w}_{\theta}\\
        %&+p_{\eta}\left(\psi-\sina{\beta}\right)\eta^2-p_{\eta}\psi\left(\psi-\sina{\beta}\right)\eta\widetilde{v}_{fi}+p_{\eta}\sina{\beta}\eta\widetilde{v}_n\\
        %&+p_i\widetilde{v}_{fi}\eta-p_i\psi\widetilde{v}_{fi}^2-\widetilde{\theta}_g\dot{\widehat{\theta}}_g+p_{v\eta}\left(\psi-\sina{\beta}\right)\widetilde{v}_n\eta\\
        %&-p_{v\eta}\psi\left(\psi-\sina{\beta}\right)\widetilde{v}_n\widetilde{v}_{fi}+p_{v\eta}\sina{\beta}\widetilde{v}_n^2-p_{v\eta}k\frac{\widetilde{\theta}_g}{\widehat{\theta}_g}\eta\widetilde{v}_n\\
        %&-p_{v\eta}k\eta\widetilde{v}_n-p_{v\eta}k_{p,v}\sina{\sigma_g}\eta^2-p_{v\eta}\eta\widetilde{w}_{\theta}+p_{vi}\widetilde{v}_n\widetilde{\eta}\\
%&-p_{vi}\psi\widetilde{v}_n\widetilde{v}_{fi}-p_{vi}k\frac{\widetilde{\theta}_g}{\widehat{\theta}_g}\widetilde{v}_{fi}\widetilde{v}_n-p_{vi}k\widetilde{v}_{fi}\widetilde{v}_n-p_{vi}k_{p,v}\sina{\sigma_g}\widetilde{v}_{fi}\eta\\
%&-p_{vi}\widetilde{v}_{fi}\widetilde{w}_{\theta}+p_{\eta i}\eta^2-p_{\eta i}\psi\eta\widetilde{v}_{fi}+p_{\eta i}\left(\psi-\sina{\beta}\right)\widetilde{v}_{fi}\eta\\
%&-p_{\eta i}\psi\left(\psi-\sina{\beta}\right)\widetilde{v}_{fi}^2+p_{\eta i}\sina{\beta}\widetilde{v}_{fi}\widetilde{v}_n.
   \end{align*}
%\begin{align}\label{eqn:after Inserting adapt law}
%    \dot{V}&= -p_vk\widetilde{v}_n^2-p_vk_{p,v}\sina{\sigma_g}\widetilde{v}_n\eta-p_v\widetilde{v}_n\widetilde{w}_{\theta}\\\nonumber
        %&+p_{\eta}\left(\psi-\sina{\beta}\right)\eta^2-p_{\eta}\psi\left(\psi-\sina{\beta}\right)\eta\widetilde{v}_{fi}+p_{\eta}\sina{\beta}\eta\widetilde{v}_n\\\nonumber
        %&+p_i\widetilde{v}_{fi}\eta-p_i\psi\widetilde{v}_{fi}^2+p_{v\eta}\left(\psi-\sina{\beta}\right)\widetilde{v}_n\eta\\\nonumber
        %&-p_{v\eta}\psi\left(\psi-\sina{\beta}\right)\widetilde{v}_n\widetilde{v}_{fi}+p_{v\eta}\sina{\beta}\widetilde{v}_n^2\\\nonumber
        %&-p_{v\eta}k\eta\widetilde{v}_n-p_{v\eta}k_{p,v}\sina{\sigma_g}\eta^2-p_{v\eta}\eta\widetilde{w}_{\theta}+p_{vi}\widetilde{v}_n\widetilde{\eta}\\\nonumber
%&-p_{vi}\psi\widetilde{v}_n\widetilde{v}_{fi}-p_{vi}k\widetilde{v}_{fi}\widetilde{v}_n-p_{vi}k_{p,v}\sina{\sigma_g}\widetilde{v}_{fi}\eta\\\nonumber
%&-p_{vi}\widetilde{v}_{fi}\widetilde{w}_{\theta}+p_{\eta i}\eta^2-p_{\eta i}\psi\eta\widetilde{v}_{fi}+p_{\eta i}\left(\psi-\sina{\beta}\right)\widetilde{v}_{fi}\eta\\\nonumber
%&-p_{\eta i}\psi\left(\psi-\sina{\beta}\right)\widetilde{v}_{fi}^2+p_{\eta i}\sina{\beta}\widetilde{v}_{fi}\widetilde{v}_n.
%\end{align}
Define $\widetilde{w}_v\triangleq p_v\widetilde{w}_{\sina{\sigma}}$, $\widetilde{w}_{\eta}\triangleq p_{v\eta}\widetilde{w}_{\sina{\sigma}}$, $\widetilde{w}_i\triangleq p_{vi}\widetilde{w}_{\sina{\sigma}}$, and inserting the adaptation law given in~\eqref{eqn:adaptive law}, and then
        %\begin{align*}
                %\dot{V}&= -p_vk\widetilde{v}_n^2-p_vk_{p,v}\sina{\sigma_g}\widetilde{v}_n\eta-\widetilde{v}_n\widetilde{w}_v\\\nonumber
        %&+p_{\eta}\left(\psi-\sina{\beta}\right)\eta^2-p_{\eta}\psi\left(\psi-\sina{\beta}\right)\eta\widetilde{v}_{fi}+p_{\eta}\sina{\beta}\eta\widetilde{v}_n\\\nonumber
        %&+p_i\widetilde{v}_{fi}\eta-p_i\psi\widetilde{v}_{fi}^2+p_{v\eta}\left(\psi-\sina{\beta}\right)\widetilde{v}_n\eta\\\nonumber
        %&-p_{v\eta}\psi\left(\psi-\sina{\beta}\right)\widetilde{v}_n\widetilde{v}_{fi}+p_{v\eta}\sina{\beta}\widetilde{v}_n^2\\\nonumber
        %&-p_{v\eta}k\eta\widetilde{v}_n-p_{v\eta}k_{p,v}\sina{\sigma_g}\eta^2-\eta\widetilde{w}_{\eta}+p_{vi}\widetilde{v}_n\widetilde{\eta}\\\nonumber
%&-p_{vi}\psi\widetilde{v}_n\widetilde{v}_{fi}-p_{vi}k\widetilde{v}_{fi}\widetilde{v}_n-p_{vi}k_{p,v}\sina{\sigma_g}\widetilde{v}_{fi}\eta\\\nonumber
%&-\widetilde{v}_{fi}\widetilde{w}_{i}+p_{\eta i}\eta^2-p_{\eta i}\psi\eta\widetilde{v}_{fi}+p_{\eta i}\left(\psi-\sina{\beta}\right)\widetilde{v}_{fi}\eta\\\nonumber
%&-p_{\eta i}\psi\left(\psi-\sina{\beta}\right)\widetilde{v}_{fi}^2+p_{\eta i}\sina{\beta}\widetilde{v}_{fi}\widetilde{v}_n.
        %\end{align*}
        adding and subtracting $\gamma_v^2\widetilde{w}_v^2-\widetilde{v}_n^2$, $\gamma_{\eta}^2\widetilde{w}_{\eta}^2-\eta^2$, $\gamma_i^2\widetilde{w}_i^2-\widetilde{v}_{fi}^2$ yields
        \begin{align*}
      %      \dot{V}=& -p_vk\widetilde{v}_n^2-p_vk_{p,v}\sina{\sigma_g}\widetilde{v}_n\eta-\widetilde{v}_n\widetilde{w}_v-\gamma_v^2\widetilde{v}_n^2+\widetilde{v}_n^2\\\nonumber
        %&+p_{\eta}\left(\psi-\sina{\beta}\right)\eta^2-p_{\eta}\psi\left(\psi-\sina{\beta}\right)\eta\widetilde{v}_{fi}+p_{\eta}\sina{\beta}\eta\widetilde{v}_n\\\nonumber
        %&+p_i\widetilde{v}_{fi}\eta-p_i\psi\widetilde{v}_{fi}^2+p_{v\eta}\left(\psi-\sina{\beta}\right)\widetilde{v}_n\eta\\\nonumber
        %&-p_{v\eta}\psi\left(\psi-\sina{\beta}\right)\widetilde{v}_n\widetilde{v}_{fi}+p_{v\eta}\sina{\beta}\widetilde{v}_n^2\\\nonumber
        %&-p_{v\eta}k\eta\widetilde{v}_n-p_{v\eta}k_{p,v}\sina{\sigma_g}\eta^2-\eta\widetilde{w}_{\eta}-\gamma_{\eta}^2\widetilde{w}_{\eta}^2+\eta^2\\\nonumber
%&+p_{vi}\widetilde{v}_n\widetilde{\eta}-p_{vi}\psi\widetilde{v}_n\widetilde{v}_{fi}-p_{vi}k\widetilde{v}_{fi}\widetilde{v}_n-p_{vi}k_{p,v}\sina{\sigma_g}\widetilde{v}_{fi}\eta\\\nonumber
%&-\widetilde{v}_{fi}\widetilde{w}_{i}-\gamma_i^2\widetilde{w}_i^2+\widetilde{v}_{fi}^2+p_{\eta i}\eta^2-p_{\eta i}\psi\eta\widetilde{v}_{fi}\\\nonumber
%&+p_{\eta i}\left(\psi-\sina{\beta}\right)\widetilde{v}_{fi}\eta-p_{\eta i}\psi\left(\psi-\sina{\beta}\right)\widetilde{v}_{fi}^2+p_{\eta i}\sina{\beta}\widetilde{v}_{fi}\widetilde{v}_n\\
%&+\gamma_v^2\widetilde{w}_v^2-\widetilde{v}_n^2+\gamma_{\eta}^2\widetilde{w}_{\eta}^2-\eta^2+\gamma_i^2\widetilde{w}_i^2-\widetilde{v}_{fi}^2\\
\dot{V}=& -p_vk\widetilde{v}_n^2-p_vk_{p,v}\sina{\sigma_g}\widetilde{v}_n\eta-\gamma_v^2\left(\widetilde{w}_v+\frac{\widetilde{v}_n}{2\gamma_v^2}\right)^2+\frac{1}{4\gamma_v^2}\widetilde{v}_n^2+\widetilde{v}_n^2\\
&+p_{\eta}\left(\psi-\sina{\beta}\right)\eta^2-p_{\eta}\psi\left(\psi-\sina{\beta}\right)\eta\widetilde{v}_{fi}+p_{\eta}\sina{\beta}\eta\widetilde{v}_n\\\nonumber
        &+p_i\widetilde{v}_{fi}\eta-p_i\psi\widetilde{v}_{fi}^2+p_{v\eta}\left(\psi-\sina{\beta}\right)\widetilde{v}_n\eta\\\nonumber
        &-p_{v\eta}\psi\left(\psi-\sina{\beta}\right)\widetilde{v}_n\widetilde{v}_{fi}+p_{v\eta}\sina{\beta}\widetilde{v}_n^2-p_{v\eta}k\eta\widetilde{v}_n\\\nonumber
        &-p_{v\eta}k_{p,v}\sina{\sigma_g}\eta^2-\gamma_{\eta}^2\left(\widetilde{w}_{\eta}+\frac{\eta}{2\gamma_{\eta}^2}\right)^2+\frac{1}{4\gamma_{\eta}^2}\eta^2+\eta^2\\
        &+p_{vi}\widetilde{v}_n\widetilde{\eta}-p_{vi}\psi\widetilde{v}_n\widetilde{v}_{fi}-p_{vi}k\widetilde{v}_{fi}\widetilde{v}_n-p_{vi}k_{p,v}\sina{\sigma_g}\widetilde{v}_{fi}\eta\\
        &-\gamma_i^2\left(\widetilde{w}_i+\frac{\widetilde{v}_{fi}}{2\gamma_i^2}\right)^2+\frac{1}{4\gamma_i^2}\widetilde{v}_{fi}^2+\widetilde{v}_{fi}^2+p_{\eta i}\eta^2-p_{\eta i}\psi\eta\widetilde{v}_{fi}\\
        &-p_{\eta i}\psi\left(\psi-\sina{\beta}\right)\widetilde{v}_{fi}^2+p_{\eta i}\sina{\beta}\widetilde{v}_{fi}\widetilde{v}_n\\
        &+\gamma_v^2\widetilde{w}_v^2-\widetilde{v}_n^2+\gamma_{\eta}^2\widetilde{w}_{\eta}^2-\eta^2+\gamma_i^2\widetilde{w}_i^2-\widetilde{v}_{fi}^2\\
        \leq &-p_vk\widetilde{v}_n^2-p_vk_{p,v}\sina{\sigma_g}\widetilde{v}_n\eta+\frac{1}{4\gamma_v^2}\widetilde{v}_n^2+\widetilde{v}_n^2\\
        &+p_{\eta}\left(\psi-\sina{\beta}\right)\eta^2-p_{\eta}\psi\left(\psi-\sina{\beta}\right)\eta\widetilde{v}_{fi}+p_{\eta}\sina{\beta}\eta\widetilde{v}_n\\
        &+p_i\widetilde{v}_{fi}\eta-p_i\psi\widetilde{v}_{fi}^2+p_{v\eta}\left(\psi-\sina{\beta}\right)\widetilde{v}_n\eta\\\nonumber
        &-p_{v\eta}\psi\left(\psi-\sina{\beta}\right)\widetilde{v}_n\widetilde{v}_{fi}+p_{v\eta}\sina{\beta}\widetilde{v}_n^2-p_{v\eta}k\eta\widetilde{v}_n\\
        &-p_{v\eta}k_{p,v}\sina{\sigma_g}\eta^2+\frac{1}{4\gamma_{\eta}^2}\eta^2+\eta^2\\
        &+p_{vi}\widetilde{v}_n\widetilde{\eta}-p_{vi}\psi\widetilde{v}_n\widetilde{v}_{fi}-p_{vi}k\widetilde{v}_{fi}\widetilde{v}_n-p_{vi}k_{p,v}\sina{\sigma_g}\widetilde{v}_{fi}\eta\\
        &+\frac{1}{4\gamma_i^2}\widetilde{v}_{fi}^2+\widetilde{v}_{fi}^2+p_{\eta i}\eta^2-p_{\eta i}\psi\eta\widetilde{v}_{fi}\\
        &-p_{\eta i}\psi\left(\psi-\sina{\beta}\right)\widetilde{v}_{fi}^2+p_{\eta i}\sina{\beta}\widetilde{v}_{fi}\widetilde{v}_n\\&+\gamma_v^2\widetilde{w}_v^2-\widetilde{v}_n^2+\gamma_{\eta}^2\widetilde{w}_{\eta}^2-\eta^2+\gamma_i^2\widetilde{w}_i^2-\widetilde{v}_{fi}^2.
        \end{align*}
        Define $\mathbf{e}_F\triangleq \begin{bmatrix}
            \widetilde{v}_n&\eta&\widetilde{v}_{fi}&\widetilde{\sina{\sigma}}_g
        \end{bmatrix}^\top$, $\mathbf{e}_R\triangleq\left[\begin{array}{c|c}
             I & \mathbf{0} 
\end{array}\right]\mathbf{e}_F$, with $I\in\mathbb{R}^{3\times3}$, is a unitary matrix, $\widetilde{\mathbf{w}}\triangleq\begin{bmatrix}
\widetilde{w}_v&\widetilde{w}_{\eta}&\widetilde{w}_i
\end{bmatrix}^\top$, $\Gamma\triangleq\begin{bmatrix}
\gamma_v^2&0&0\\
0&\gamma_{\eta}^2&0\\
0&0&\gamma_i^2
\end{bmatrix}$, and using~\eqref{eqn:matrix Q} yields
\begin{align}\label{eqn:b4dissIneq}
    \dot{V}\leq-\mathbf{e}_F^\top Q_F \mathbf{e}_F+\widetilde{\mathbf{w}}^\top \Gamma \widetilde{\mathbf{w}}-\mathbf{e}_R^\top\mathbf{e}_R,
\end{align}
then adding and subtracting $\alpha V$ yields 
\begin{align}\label{eqn:bfdissineq}
    \dot{V}\leq& -\mathbf{e}_F^\top\left(Q_F-\alpha P_F\right)\mathbf{e}_F-\alpha V+\widetilde{\mathbf{w}}^\top\Gamma\widetilde{\mathbf{w}}-\mathbf{e}_R^\top\mathbf{e}_R\nonumber\\
   % \leq&-\mathbf{e}_R^\top\left(Q-\alpha P\right)\mathbf{e}_R+\widetilde{\mathbf{w}}^\top\Gamma\widetilde{\mathbf{w}}-\mathbf{e}_R^\top\mathbf{e}_R\\
    \leq&-\mathbf{e}_R^\top\left(Q-\alpha P\right)\mathbf{e}_R+\overline{\gamma}^2\widetilde{\mathbf{w}}^\top\widetilde{\mathbf{w}}-\mathbf{e}_R^\top\mathbf{e}_R.
\end{align}
Applying the sufficient conditions in the LMI from the optimization problem ~\eqref{eqn:optimization prob.} leads to
\begin{align}\label{eqn:dissIneq}
    \dot{V}\leq\overline{\gamma}^2\left\|\widetilde{\mathbf{w}}\right\|^2-\left\|\mathbf{e}_R\right\|^2,
\end{align}
from which the result follows.
$\blacksquare$
\end{document}